\documentclass[a4paper,10pt]{article}

\usepackage[margin=3cm]{geometry}
\usepackage[normalem]{ulem}
\usepackage{amssymb,amsbsy,amsmath,amsfonts,amssymb,amscd,amsthm}
\usepackage{xcolor}
\usepackage{booktabs}
\usepackage{bbm}
\usepackage{url}
\usepackage{float}
\usepackage[ colorlinks=true, linkcolor=black, citecolor=black, urlcolor=black ]{hyperref} 
\definecolor{bibblue}{RGB}{0,0,180} \newcommand{\biblink}[2]{
 \href{#1}{{\color{bibblue}#2}}
}

\newcommand{\R}{\mathbb{R}}

\newcommand{\E}{\mathbb{E}}
\newcommand{\bS}{\mathbb{S}}

\usepackage{graphicx}
\graphicspath{{./}{figures/}}
\usepackage{algorithm}
\usepackage{algpseudocode}

\newcommand{\beq}{\begin{equation}}
\newcommand{\eeq}{\end{equation}}
\def\a{\alpha}

\def\g{\gamma}

\def\n{\nu}

\def\e{\varepsilon}

\def\pd{\partial}

\newcommand{\cN}{{\cal N}}

\newcommand{\cL}{{\cal L}}

\newcommand{\diver}{{\rm div}}
\newcommand{\Var}{{\rm Cov}}

\newtheorem{theorem}{Theorem}[section]

\newtheorem{proposition}[theorem]{Proposition}

\newtheorem{remark}[theorem]{Remark}

\numberwithin{equation}{section}

\titlepage
\title{A Mean Field Games Perspective on Evolutionary Clustering}
\date{}
		
\author{Alessio Basti\footnote{Dip. di Ingegneria e Geologia, Univ. ``G. D'Annunzio'' Chieti-Pescara,
		viale Pindaro 42, 65127 Pescara (Italy), { \textsc alessio.basti@unich.it, fabio.camilli@unich.it.}} \and
 Fabio Camilli$^\ast$ \and 
		Adriano Festa \footnote{Dip. di Scienze Matematiche, Politecnico di Torino,
		corso Duca degli Abruzzi , 24, 10129 Torino (Italy), { \textsc adriano.festa@polito.it}}}

\begin{document}
\maketitle

\begin{abstract}
We propose a control-theoretic framework for evolutionary clustering based on quasi-stationary Mean Field Games. Each cluster is represented by a probability density whose evolution is governed by a Fokker--Planck equation, while the associated velocity field is determined through a stationary Hamilton--Jacobi equation. The general formulation does not prescribe a finite-dimensional statistical shape for the component densities, although the number of components is fixed. In the Gaussian specialization, we show that suitable affine dynamics reproduce the mean and covariance trajectories generated by the classical Expectation--Maximization procedure. To improve temporal coherence in the presence of noise and temporary cluster overlaps, we introduce causal and non-causal time-averaged log-likelihood objectives. We also develop a fully density-based numerical implementation for non-Gaussian components. The proposed formulations are assessed on synthetic and real time-dependent datasets and compared with independent snapshot Expectation--Maximization
and with the same method applied to temporally smoothed observations. In the two-dimensional benchmark, an established
evolutionary \(k\)-means method is also included as an external dynamic-clustering baseline.
\end{abstract}


\section{Introduction}

Clustering is a central task in unsupervised learning, where the objective is to identify coherent subpopulations
within a dataset: observations assigned to the same group should exhibit a high degree of similarity, whereas
observations belonging to different groups should remain distinguishable. A classical probabilistic approach is based
on finite mixture models, in which the observed distribution is represented as a weighted combination of component
distributions, most commonly Gaussian ones. The corresponding parameters are frequently estimated through the
Expectation--Maximization (EM) algorithm, which provides both a soft assignment of the observations to the
components and a probabilistic interpretation of the resulting clusters; see, for example, \cite{Bishop,bilmes}.

Most classical clustering techniques are designed for static datasets. In many applications, however, the observed
distribution changes over time, and the relevant task is not only to identify clusters at individual time instants but also to
track their evolution while preserving temporal coherence. Examples arise whenever populations, spatial distributions,
or patterns of activity evolve over successive observations. Applying a static clustering method independently at
each time may lead to unstable assignments, abrupt changes in the inferred component parameters, or exchanges of
component labels when different populations approach or overlap.

Several evolutionary and dynamic clustering approaches have been proposed to address these issues, including
repeated clustering with temporal penalties, evolutionary regularization, and methods that combine snapshot accuracy
with historical consistency \cite{aynaud2010static,chakrabarti2006,chi2007evolutionary,xu2014dynamic}.
The literature has broadened considerably in recent years. A recent review of clustering evolution in
temporal and streaming data distinguishes, among others, evolutionary clustering, self-organizing approaches,
heuristic procedures, and online stream-clustering methods \cite{atif2023review}. Within the classical evolutionary
clustering paradigm, temporal coherence is introduced by balancing the quality of the current partition against a
history-dependent cost, as in Evolutionary \(k\)-means \cite{chakrabarti2006}, or by adding temporal smoothness to
spectral clustering objectives \cite{chi2007evolutionary}. Other approaches are tailored to specific data structures, such
as dynamic stochastic block models for evolving networks \cite{xu2014dynamic}.

More recent developments place increasing emphasis on non-stationarity and concept drift. For example,
StreamDD combines stream clustering with explicit drift detection \cite{hamdad2024streamdd}, while STM-Stream uses
a sliding-window short-term memory mechanism to preserve information from previous data distributions while adapting
to evolving streams \cite{liang2025stmstream}. In a spatial--temporal setting, Portugal et al.~\cite{portugal2024graph}
construct clusters at successive timestamps and represent their evolution through a graph linking related clusters.
Recent work has also emphasized the need for validation criteria that remain meaningful under concept drift
\cite{iglesias2024temporal}, while the dynamic-community literature has increasingly incorporated heuristic,
metaheuristic, deep-learning, and hybrid methods for time-evolving graphs; see the recent review
\cite{saeed2025dynamic}.

These families address different versions of the dynamic clustering problem and are not interchangeable.
The approaches most directly related to evolutionary clustering are typically formulated through discrete-time
snapshot updates, temporal penalties, online summaries, or explicit links between partitions obtained at successive
times. Stream-clustering methods additionally emphasize one-pass processing, limited memory, concept drift, and in
many cases the appearance or disappearance of clusters. Such features are important but differ from the setting
considered here, where the observed object is a time-dependent density \(f(\cdot,t)\) and the number of latent
components is fixed. The aim is therefore not to replace existing stream- or graph-clustering techniques, but to provide
a complementary continuous-density formulation in which temporal evolution is part of the clustering model itself.

Motivated by recent connections between machine learning, dynamical systems, differential geometry,
and optimal control \cite{chaudari,chen,e,hl,hpv,mc}, evolutionary clustering is investigated here from the perspective
of Mean Field Games (MFG).
Mean Field Games were introduced independently by Lasry and Lions \cite{ll} and by Huang et al.~\cite{hmc} to
describe the collective behaviour of a large number of interacting agents. In the mean-field limit, the optimization
problem of a representative agent is encoded by a Hamilton--Jacobi equation, while the evolution of the population
distribution is governed by a Fokker--Planck equation. The coupling between these equations describes the interaction
between individual decisions and the macroscopic population state. Beyond their original applications in economics
and control, mean-field models have increasingly been connected with machine learning, optimal transport, and
large-scale optimization, where probability distributions evolve according to variational or control-theoretic
principles \cite{accd,clz,cf,pequito,zk}.

Mean-field and MFG ideas have already been used in clustering. Pequito et al.~\cite{pequito} proposed
an MFG formulation for unsupervised learning of finite mixture models, Aquilanti et al.~\cite{accd} developed a
multi-population MFG approach to cluster analysis that can be interpreted as a continuous version of EM, and Herty
et al.~\cite{hpv} derived mean-field clustering models from interacting-particle dynamics. In these works, the
continuous dynamics primarily provide an algorithmic or population-dynamical mechanism for solving a clustering
problem associated with a given dataset. The present problem is different: the observed density itself varies with the
physical time \(t\), and the objective is to follow the latent components coherently along this prescribed temporal
sequence. Thus, the novelty is not the use of a continuous model for clustering per se, but its use as an evolutionary
clustering mechanism for a time-dependent observed density.

Within this framework, evolutionary clustering is interpreted as a controlled evolution of several interacting
probability densities. Each component of the mixture represents a population whose density changes through a
Fokker--Planck equation, while its velocity field is obtained from an associated stationary Hamilton--Jacobi problem.
The observed time-dependent density enters the model through data-fidelity and assignment terms, whereas the
controlled dynamics provide a continuous mechanism for propagating the cluster densities in time.
In contrast with temporal penalties acting only on successive partitions or centroids, temporal
coherence is therefore encoded directly in the evolution equation for each component density, while the associated
Hamilton--Jacobi problem determines the velocity field through an optimality principle.

The general formulation is non-parametric with respect to the shapes of the component densities: no prescribed
finite-dimensional statistical family is imposed on them, although the number of components is assumed to be fixed.
This distinction is important. The paper does not propose a non-parametric estimator of the number of clusters;
rather, it develops a density-based evolution framework in which the individual components may assume non-Gaussian
shapes. The Gaussian setting is subsequently studied as a tractable specialization that permits a precise comparison
with classical mixture-model estimation.

More specifically, a quasi-stationary MFG formulation is considered.
At each physical time, agents optimize with
respect to the current population configuration rather than anticipating its complete future evolution. This assumption
leads to a sequence of stationary control problems coupled with time-dependent Fokker--Planck equations and is
particularly suitable for sequentially observed or slowly varying data. In the Gaussian case, the quadratic value-function
ansatz generates affine velocity fields. By selecting their coefficients through suitable moment-matching conditions,
the resulting Fokker--Planck dynamics recover the means and covariance matrices associated with the corresponding
EM updates. The role of this specialization is therefore twofold: it provides an analytically accessible validation of
the MFG construction and establishes a continuous-time connection with a standard clustering methodology.

Independent snapshot estimation remains sensitive to temporary overlaps, observational noise, and rapid variations
of the data. To improve temporal coherence, time-averaged log-likelihood objectives are also introduced.
A causal asymmetric kernel averages information from past observations and can therefore be used in sequential or
online settings. A symmetric kernel incorporates observations from both sides of the current time and is instead intended
for offline reconstruction and post-processing. These formulations regularize the inferred weights, centres, and covariance
matrices directly, rather than merely smoothing the observed density before applying a static clustering procedure.

The main contributions of the present work can be summarized as follows:
\begin{itemize}
    \item Evolutionary clustering is formulated as a quasi-stationary multi-population MFG system
    coupling stationary Hamilton--Jacobi equations with time-dependent Fokker--Planck equations.

    \item A Gaussian specialization is derived in which affine controlled dynamics reproduce the EM
    moment trajectories, thereby providing a continuous population-dynamics interpretation of the classical estimation
    procedure.

    \item Causal and non-causal time-averaged log-likelihood formulations are introduced to reduce
    abrupt variations and component exchanges during temporary overlaps.

    \item Both a finite-dimensional Gaussian implementation and a fully density-based numerical method
    are developed, the latter allowing component shapes that are not prescribed in advance.

    \item The proposed formulations are compared with independent snapshot EM, EM applied to
    temporally smoothed observations, and an established Evolutionary \(k\)-means baseline
    \cite{chakrabarti2006}, assessing reconstruction accuracy, component tracking, temporal regularity, and
    computational cost.

    \item The approach is tested on non-Gaussian synthetic examples, a two-dimensional evolutionary
    clustering benchmark, and a real time-dependent dataset describing seasonal humpback-whale distributions.
\end{itemize}

The numerical experiments are intended to clarify the different roles of statistical estimation, temporal regularization,
and controlled density evolution. The comparisons show that smoothing the observed data may improve the
reconstruction of the total distribution, but does not by itself guarantee coherent trajectories for the individual
components. The evolutionary formulations instead regularize the latent clustering variables directly.
The two-dimensional benchmark additionally compares the proposed formulations with Evolutionary
\(k\)-means, providing a direct numerical comparison with a classical history-regularized evolutionary clustering
method.
The fully density-based MFG model is more computationally demanding than its Gaussian counterparts, but it can
represent component geometries that cannot be captured accurately by a prescribed Gaussian family. The results
therefore highlight a trade-off between computational efficiency, data fidelity, temporal coherence, and flexibility of
the component representation.

The remainder of the paper is organized as follows. Section~\ref{sec:MFG} introduces the general quasi-stationary
MFG framework for the evolution of multiple component densities. Section~\ref{sec:EM_model} studies the
Gaussian specialization, establishes its relationship with the EM moment updates, and derives the causal and symmetric
time-averaged formulations. Section~\ref{sec:num_examples} presents the numerical methods and comparative experiments,
including non-Gaussian synthetic densities, the two-dimensional benchmark, and the real-data application. The final
section summarizes the main findings, discusses the limitations of the proposed formulations, and outlines directions
for further theoretical and computational developments.

\section{Non-Parametric Mean Field Games Framework}\label{sec:MFG}
We represent the dataset at time $t \in [0,T]$ as a probability density $f(\cdot,t)$ mapping $\mathbb{R}^d$ into $(0,\infty)$.  The mixture $m$ is composed of $K$ components expressed as
\begin{equation}\label{emmfg_mix}
m(x,t) = \sum_{k=1}^K \alpha_k(t) m_k(x,t), \qquad \alpha(t) = (\alpha_1(t), \dots, \alpha_K(t)) \in \mathcal T_K,
\end{equation}
where $m_k$ are the cluster densities and $\mathcal T_K$ is the unit simplex. The weights $\alpha_k(t)$ describe the proportion of the total population associated with the $k$-th component at time $t$. 
We highlight that the reference density $f(x,t)$ is assumed to be given and represents the observed data distribution at time $t$, while the mixture $m(x,t)$ approximates $f$ by evolving according to controlled dynamics.\\
By Bayes' theorem, we introduce the probabilities $\gamma_k$ (also known as \textit{responsibilities}) that each data point belongs to a given latent component $k$ of the mixture, indicating how much each component contributes to generating that observation, i.e. 
\begin{equation}\label{emmfg_resp}
    \gamma_k(x,t)=\dfrac{\alpha_k(t) m_k(x,t)}{m(x,t)},\qquad x\in\mathbb{R}^d,\,t\in [0,T],\, k=1,\dots,K.
\end{equation}
These responsibilities can be used to divide the dataset into clusters by setting:
\begin{equation*} 
    C_k(t) = \left\{ x \in \mathbb{R}^d \;\middle|\; \gamma_k(x,t) = \max_{j=1,\dots,K} \gamma_j(x,t) \right\}.
\end{equation*}
To compute the densities $m_k$ and the weights $\alpha_k$, we consider a \emph{quasi-stationary} MFG. In contrast to standard MFGs, where agents anticipate the future evolution of the population, the quasi-stationary setting assumes agents react only to the current population state. This makes the model particularly suitable for online or slowly varying data (see \cite{mouz}).
In this setting, each data point is viewed as a controlled stochastic agent. The state $X_t$ of an agent in the $k$-th cluster evolves according to:
\begin{equation}\label{eq:dynamics}
dX_t =a(X_t)\,dt + \sqrt{2\varepsilon}\,dW_t, \quad X_0 = x,
\end{equation}
where $a$ is the control, $\varepsilon > 0$ is the diffusion coefficient, and $W_t$ denotes a standard Wiener process. We assume that, at each time instant $t$, the agent belonging to the $k$-th population optimizes a cost functional $J_k(x,a;t)$ using the information available at the present time, where
\begin{equation}\label{eq:ergodic_cost}
J_k(x,a;t) = \lim_{S \to +\infty} \frac{1}{S} \mathbb{E}_x \left[ \int_t^S \left( L(\tilde X_s, a(\tilde X_s)) + F_k(\tilde X_s, m(\cdot,t)) \right) ds \right],
\end{equation}
where $L$ is a Lagrangian cost, and $F_k$ is the interaction term coupling the agent's state with the distribution of the population, $\tilde X$ refers to a (fictitious) dynamics of type \eqref{eq:dynamics} starting at time $t$. The optimal control $a$ obtained by optimizing this functional is then incorporated into dynamics \eqref{eq:dynamics} \cite{mouz}. Note that, in \eqref{eq:ergodic_cost}, the distribution $m$ is frozen at the current time  $t$. Therefore, agents do not anticipate its future evolution, but instead optimize their configuration based solely on the information available at that time.
\\ 
To bridge the microscopic dynamics with the macroscopic description, we introduce the Hamiltonian $H$ via the Legendre transform of the Lagrangian:
\begin{equation*}
H(x, p) = \sup_{a \in \mathbb{R}^d} \left\{ -p \cdot a - L(x, a) \right\}.
\end{equation*}
Under the quasi-stationary assumption, agents solve a stationary optimization problem at each time $t$, while the densities $m_k$ evolve according to the optimal drift $-\partial_p H(x, Du_k)$, where the notation $\partial_p H$ denotes the derivative of $H$ with respect to the momentum $p$. The resulting coupled system is given by:
\begin{equation}\label{eq:MFG_full}
\left\{
\begin{array}{ll}
-\varepsilon\Delta u_k(x,t) + H\big(x,Du_k(x,t)\big) + \lambda_k(t) = F_k\big(x,m(\cdot,t)\big),\\[6pt]
\partial_t m_k(x,t) - \varepsilon\Delta m_k(x,t) - \operatorname{div}\!\big(m_k(x,t)\,\partial_p H(x,Du_k(x,t))\big)=0,\\[6pt]
\alpha_k(t) = \int_{\mathbb{R}^d} \gamma_k(x,t)\,f(x,t)\,dx,\\[6pt]
m_k(x,0)=m_{k,0}(x),\qquad \min_{x \in \R^d} u_k(x,t)=0,
\end{array}
\right\}_{\forall\,k=1,\dots,K}
\end{equation}
where $u_k$ is the ergodic value function, and $\lambda_k(t)\in\mathbb{R}$ is the ergodic constant. It is worth noticing that the responsibilities $\gamma_k$ couple the sub-populations to the empirical data density $f$ and to the full mixture $m$.

The coupling $F_k$ determines the qualitative behavior of the clustering dynamics and can encode, for example, entropy regularization, data fidelity of the mixture $m$ to the observed data $f$, and repulsion/attraction between components. A meaningful information-theoretic choice is a convex combination of an entropy penalization and a Kullback-Leibler-type fidelity term:
\begin{equation} \label{eq_KL}
F_k(x,m) = \underbrace{\eta_k \log\big(m_k(x,t) \big)}_{\text{Entropic penalization}} + \underbrace{(1-\eta_k) \alpha_k(t) \log\left( \frac{m(x,t)}{f(x,t)} \right)}_{\text{Data fidelity}},\quad \eta_k\in[0,1]. 
\end{equation}
Specifically, the terms in $F_k$ correspond to the first variational derivative with respect to $m_k$ of the negative Shannon entropy $\mathcal{H}(m_k) = \int m_k \log(m_k) dx$ and the Kullback-Leibler divergence $\text{KL}(m \| f) = \int m \log(m/f) dx$. These derivatives are considered up to additive constants, which are absorbed by the ergodic constant $\lambda_k(t)$ in the Hamilton-Jacobi-Bellman equation of system \eqref{eq:MFG_full}.

We do not address the existence of solutions to the nonlinear system \eqref{eq:MFG_full}, as our focus is primarily on
computational aspects. Nevertheless, such a result could, in principle, be
established via a fixed-point argument (cf. \cite{mouz}). All the subsequent analytical statements concerning system \eqref{eq:MFG_full} are therefore understood conditionally on the existence of sufficiently regular solutions.

\section{Fokker--Planck Equation and the Gaussian Case}\label{sec:EM_model}
To bridge our MFG perspective with the standard Gaussian EM algorithm, we consider a specialized quadratic coupling term that acts as a dynamic attractor:
\begin{equation*} 
    F_k(x, m)=\frac{1}{2} (x-M_k(t))^\top V_k^2(t)(x-M_k(t)),
\end{equation*}
for a suitable positive-definite symmetric matrix $V_k(t)$. For the quadratic Hamiltonian $H(x,p)=\tfrac{1}{2}|p|^2$, the ansatz
\begin{equation}\label{eq:u_quadratic}
u_k(x,t)=\tfrac{1}{2}\,(x-M_k(t))^\top V_k(t)\,(x-M_k(t))
\end{equation}
gives the linear drift
\begin{equation}\label{eq:drift}
B_k[m](x,t):=\partial_p H(x,Du_k)=Du_k=V_k(t)\big(x-M_k(t)\big).
\end{equation}
With this drift, the second equation in \eqref{eq:MFG_full} becomes a Fokker--Planck (FP) equation with affine drift.  The MFG approach thus gives rise to the following system
\begin{equation}\label{emMFG}
\left\{
\begin{array}{ll}
		\pd_t m_k(x,t)-	\e    \Delta m_k(x,t)-\diver(m_k(x,t)B_k[m](x,t) )=0, \\[8pt]
		\a_k(t)=\int_{\R^d }\g_k(x,t) f(x,t)dx, \\[8pt]
	m_k(x,0)=m_{k,0} (x)
\end{array}
\right\}_{\forall\,k=1,\dots,K}
\end{equation}
for $(x,t)\in\R^d\times [0,T]$.  For mathematical consistency, we assume finite first and second moments:
\[
\int_{\mathbb{R}^d}|x|\,f(x,t)\,dx \le C_f, \qquad \int_{\mathbb{R}^d}|x|^2\,f(x,t)\,dx \le C_f,
\]
for a constant $C_f > 0$.

To clarify better the procedure, substituting the quadratic ansatz into the stationary Hamilton--Jacobi equation shows that it is satisfied for $\lambda_k(t)=\varepsilon\operatorname{tr}V_k(t)$. Thus, in the Gaussian specialization, the Hamilton--Jacobi equation provides the control interpretation of the affine velocity field but does not determine $M_k$ and $V_k$; these quantities are selected below through the conditions that match the Gaussian expectation-maximization moments, as, in the following proposition, we prove that a solution of \eqref{emMFG} at time $t$, $m_k(\cdot,t)$,  is a Gaussian density $\cN(x|\nu_k(t), T_k(t))$ explicating the dependence of the parameter on $M_k$, $V_k$.
\begin{proposition}\label{prop: explicit_sol}
	Let $m$ in \eqref{emmfg_mix} be a mixture with $(m_k,\alpha_k)$, $k=1,\dots,K$, given by   the  system \eqref{emMFG}. Assume that, for every $k=1,\dots,K$, the initial density $m_{k,0}$ is Gaussian. Then, for each time $t$,  $m(\cdot,t)$ is a Gaussian mixture with    
	\begin{equation*} 
		m_k(x,t)= c_k(t)\,\exp\Bigl(-\frac{1}{2}(x-\nu_k(t))^\top \,T_k(t)^{-1}(x-\nu_k(t))\Bigr)
	\end{equation*}
	where
	\begin{equation*} 
		\begin{cases}
			\nu'_k(t)= -V_k(t)\Bigl(\nu_k(t)-M_k(t)\Bigr),\\[1mm]
			{T'_k(t)= 2\e\,I - T_k(t)\,V_k(t)-V_k(t)\,T_k(t)},\\[1mm]
			c_k(t)= \frac{1}{\sqrt{(2\pi)^d \det( T_k(t))}}.
		\end{cases}
	\end{equation*}
\end{proposition} 
The proof of the previous proposition can be found in the \ref{app:proofs}.\\
We define
\begin{align}
	&\mu_k(t)=\frac{\int_{\R^d} x\, \gamma_{k}(x,t) f(x,t)\, dx}{\int_{\R^d} \gamma_{k}(x,t) f(x,t)\, dx} \in \R^d, \quad & t\in [0,T], \label{emmfg_mean}\\[2mm]
	&\Sigma_k(t)=\frac{\int_{\R^d} (x-\mu_k(t))(x-\mu_k(t))^\top \gamma_{k}(x,t) f(x,t)\, dx}{\int_{\R^d} \gamma_{k}(x,t) f(x,t)\, dx} \in \bS^d_+, & t\in [0,T], \label{emmfg_var}\\[1mm]
	&m_{k,0}(x)=\mathcal{N}(x \mid \mu_k(0), \Sigma_k(0)). \label{initial_datum}
\end{align}

Here, $\mu_k(t)$ and $\Sigma_k(t)$ are the mean and covariance of the $k$-th component at time $t$, computed as weighted averages of the reference density $f(x,t)$. 
In other words, $\mu_k(t)$ is the expected value of $x$ if we focus only on component $k$, and $\Sigma_k(t)$ measures how spread out the data is around this mean. We distinguish the actual moments $(\nu_k,T_k)$ of the evolving Fokker--Planck density from the responsibility-weighted target moments $(\mu_k,\Sigma_k)$ extracted from the observations. The following results select $M_k$ and $V_k$ so that the two pairs coincide.

The choice of $m_k(x,0)$ reflects the assumption that the initial distribution $m(0)$ is itself a Gaussian mixture approximating the dataset $f(x,0)$.

In the models that follow, the vector $M_k(t)\in\R^d$ and positive definite symmetric matrix $V_k(t)$ in \eqref{eq:drift} depend on $\gamma_k(t)$, $\mu_k(t)$, and $\Sigma_k(t)$, and therefore on the overall density $m$. They are chosen to ensure that the mixture $m$ in \eqref{emmfg_mix} remains a Gaussian mixture that maximizes a log-likelihood functional based on the data. 
In the next subsections, we present two models that aim to maximize the log-likelihood while keeping each cluster stable and well-defined over time.

\subsection{Instantaneous log-likelihood model}\label{sec:first_model}
In this section, we consider a model that at each time maximizes the Gaussian log-likelihood functional  
\begin{equation}
	\label{log_lik}
	\cL(\{\beta_k, \rho_k, \Delta_k\}_{k=1}^K \,\big|\, f(t)) = \int_{\R^d} \sum_{k=1}^{K} \gamma_k(x,t) \ln\left(\beta_k  \cN(x|\rho_k , \Delta_k )\right) f(x,t)\,dx,
\end{equation}
with respect to the parameters   
$\beta_k >0$, $\sum_{k=1}^K\beta_k=1$, 
$\rho_k  \in\mathbb{R}^d$, and $\Delta_k$ (as a positive-definite symmetric matrix), where the responsibilities $\gamma_k(x,t)$ are given. The previous functional measures how well the mixture model explains the observed data and it is suitable for the application of EM algorithm (cf. \cite{Bishop}).
\begin{proposition}\label{prop:log_lik_1}
	For any $t\in [0,T]$, assume that $\alpha_k(t)>0$ and $\Sigma_k(t)\in\mathbb{S}_{++}^d$
for every \(k=1,\dots,K\). Then the unique maximizer of $\cL$  in \eqref{log_lik} is given by
	$\beta_k  =\alpha_k(t)$, $\rho_k=\mu_k(t)$ and  $\Delta_k=\Sigma_k(t)$, where $\alpha_k$, $\mu_k$ and $\Sigma_k$   are defined as in \eqref{emMFG}, \eqref{emmfg_mean} and \eqref{emmfg_var}.
\end{proposition}
The proof of the previous proposition is classical and can be found in \cite{Bishop} since the time plays no role.
In the next proposition we show that, for appropriately chosen $M_k(t)$, $V_k(t)$, the parameters of the mixture components $m_k$ of $m$ coincide at each time with the maximizer of the functional $\cL$ in \eqref{log_lik} (for the proof, see the \ref{app:proofs}). 
\begin{proposition}\label{prop:equal_data}
Let \((m_k,\alpha_k)\), \(k=1,\dots,K\), be a solution of \eqref{emMFG}. Assume that \(\mu_k\in C^1([0,T];\mathbb R^d)\), \(\Sigma_k\in C^1([0,T];\mathbb{S}_{++}^d)\), and
\[
Q_k(t):=2\varepsilon I-\Sigma_k'(t)\in\mathbb{S}_{++}^d
\qquad\text{for every }t\in[0,T].
\]
For each \(t\), let \(V_k(t)\) be the unique symmetric positive-definite solution of the Sylvester equation and define \(M_k(t)\) by
\begin{equation}\label{coeff_eq_1model}
\begin{cases}
\Sigma_k(t)V_k(t)+V_k(t)\Sigma_k(t)=2\varepsilon I-\Sigma_k'(t),\\[3pt]
M_k(t)=\mu_k(t)+V_k(t)^{-1}\mu_k'(t).
\end{cases}
\end{equation}
Then, for every \(k=1,\dots,K\),
\begin{equation}\label{equal_mean}
\nu_k(t)=\mu_k(t),\qquad T_k(t)=\Sigma_k(t)
\qquad\text{for every }t\in[0,T].
\end{equation}
Moreover,
\[
\mathbb E_{m(\cdot,t)}[X]=\mathbb E_{f(\cdot,t)}[X],
\qquad
\operatorname{Cov}_{m(\cdot,t)}[X]=\operatorname{Cov}_{f(\cdot,t)}[X].
\]
\end{proposition}

If a commutation condition \([\Sigma_k,\Sigma_k']=0\) holds, we have the closed-form expression
\[
V_k(t)=\frac12\Sigma_k(t)^{-1}\bigl(2\varepsilon I-\Sigma_k'(t)\bigr).
\]
Nevertheless, such a condition is not required. Indeed, for a positive-definite \(\Sigma_k\), the Sylvester operator \(V\mapsto\Sigma_kV+V\Sigma_k\) is invertible on symmetric matrices. Moreover, when \(Q_k\succ0\), $V_k=\int_0^\infty e^{-s\Sigma_k}Q_k e^{-s\Sigma_k}\,ds\succ0$, so \(V_k^{-1}\) in \eqref{coeff_eq_1model} is well defined. On a compact time interval, the sufficient condition \(Q_k\succ0\) can be enforced by choosing \(\varepsilon>\tfrac12\sup_t\lambda_{\max}(\Sigma_k'(t))\). The resulting moment identities remain independent of \(\varepsilon\), although the auxiliary drift parameters \(V_k,M_k\) do depend on it.
\subsection{Time-averaged log-likelihood model}\label{sec:second_model}
 The objective of this section is no longer to select the model parameters that maximize the instantaneous log-likelihood functional. Instead, we introduce a drift term that enables the mixture weights,  means, and covariances to maximize a time-averaged log-likelihood over a finite window. This approach smooths out short-term fluctuations in the observations. Specifically, we proceed following two different principles:
 \begin{itemize}
 \item (Asymmetric regularization) We assign weights to past data using a kernel function $g$, which performs a convolution over the interval $[t - \tau, t]$ for a fixed  \(\tau>0\). The kernel $g : [0,\tau] \to \mathbb{R}$ satisfies
\begin{equation}\label{eq:asym-relax}
g(\tau) = 0,
\quad
g(s) \geq 0 \quad \forall\,  s\in(0,\tau],
\quad
\int_{\R} g(s)\,ds = 1.
\end{equation}
As an explicit example of $g$, we can consider 
the following exponential kernel that assigns more weight to recent observations  (asymmetric)
\begin{equation}\label{eq:drift-asym} g(s) = \frac{e^{-s} - e^{-\tau}}{1 - e^{-\tau}(1 + \tau)} \mathbf{1}_{[0,\tau]}(s) \end{equation}
\item (Symmetric regularization) We use a standard symmetric mollifier giving the same relevance to the past and the future choosing $g$ as
\begin{equation}\label{eq:sym-relax} g(s):=\frac{2}{\tau C}
\exp\!\left(-\frac{1}{1-(2s/\tau)^2}\right)\mathbf{1}_{[-\tau/2,\tau/2]}(s), \hbox{ with }C=\int_{-1}^1 e^{-\frac{1}{1-s^2}}ds.
\end{equation}
in this case we have that 
\begin{equation*} 
g(-\tau/2)=g(\tau/2) = 0,
\quad
g(s) \geq 0 \quad \forall\,  s\in(-\tau/2,\tau/2],
\quad
\int_{\R} g(s)\,ds = 1.
\end{equation*}
\end{itemize}

Both kernels are normalized explicitly. For \eqref{eq:drift-asym},
\[
\int_0^\tau \frac{e^{-s}-e^{-\tau}}{1-e^{-\tau}(1+\tau)}\,ds
=\frac{1-e^{-\tau}-\tau e^{-\tau}}{1-e^{-\tau}(1+\tau)}=1.
\]
For \eqref{eq:sym-relax}, the substitution \(y=2s/\tau\) gives
\[
\int_{-\tau/2}^{\tau/2}g(s)\,ds
=\frac1C\int_{-1}^{1}\exp\!\left(-\frac{1}{1-y^2}\right)dy=1.
\]

Given $g$, we denote  by $*$ the standard convolution in $\R$, i.e.
\[
g *m=\int_{-\infty}^{\infty}m(s)g(t-s)ds.
\]
When applied to vector- or matrix-valued quantities, convolution is understood component-wise.
The aim is to find a Gaussian mixture $m(t)$\footnote{With a slight abuse of notation, in the time-averaged model we also denote by \(m\) the mixture
\(\sum_{k=1}^K \widetilde{\alpha}_k(t)m_k(\cdot,t)\).} given by the solution of \eqref{emMFG} which maximizes the convolution of the log-likelihood functional \eqref{log_lik} with the kernel $g$, i.e.
\begin{equation}\label{log_lik2}
	\begin{split}
		&\cL_g(\{\beta_k, \rho_k, \Delta_k\}_{k=1}^K \,\big|\, f):=\bigl(g *  \cL\bigr)(t)=\\
		&\quad=\sum_{k=1}^K
		\int_{\R}g(t-s)\int_{\R^d}\gamma_k(x,s)\ln\bigl[\beta_k\,\mathcal N(x\mid\rho_k,\Delta_k)\bigr]\,f(x,s)\,dx\,ds.
	\end{split}
\end{equation}
Note that, when the kernel $g$ coincides with the Dirac delta function $\delta_0$, then $\cL_g$ coincides with the functional $\cL$ in \eqref{log_lik}. 
To define the convolution near the temporal endpoints, all time-dependent quantities entering it are extended constantly outside $[0,T]$: by their value at $0$ on $([-\tau,0])$, and, for the symmetric kernel, by their value at $T$ on $[T,T+\tau/2]$.
\begin{proposition}\label{prop:log_lik2}
Fix \(t\) and assume $(g*\alpha_k)(t)>0$. The maximizers of \(\mathcal L_g\) with respect to \(\beta_k\) and \(\rho_k\) are
\begin{align}
\beta_k^\star(t)&=(g*\alpha_k)(t)=:\widetilde\alpha_k(t),\label{maximizers_glog1}\\
\rho_k^\star(t)&=\frac{(g*(\alpha_k\mu_k))(t)}{(g*\alpha_k)(t)}=:\widetilde\mu_k(t).\label{maximizers_glog2}
\end{align}
Assuming that the weighted covariance below is positive definite, the exact covariance maximizer is
\begin{equation}\label{eq:exact-temporal-covariance}
\Delta_k^\star(t)=\widetilde\Sigma_k(t)+R_k(t),
\end{equation}
where
\begin{equation}\label{maximizers_glog3}
\widetilde\Sigma_k(t):=
\frac{(g*(\alpha_k\Sigma_k))(t)}{(g*\alpha_k)(t)}
\end{equation}
and
\begin{equation}\label{eq:centroid-remainder}
R_k(t):=
\frac{\bigl(g*(\alpha_k(\mu_k-\widetilde\mu_k(t))(\mu_k-\widetilde\mu_k(t))^\top)\bigr)(t)}{(g*\alpha_k)(t)}.
\end{equation}
In particular, \(R_k(t)\succeq0\). Let
\[
I_t:=\{s:g(t-s)>0\},
\qquad
L_{\mu,k}(t):=\sup_{s\in I_t}\|\mu_k'(s)\|_2.
\]
Since both kernels used here have support diameter \(\tau\),
\begin{equation}\label{eq:covariance-remainder-bound}
\|R_k(t)\|_2\le \frac12 L_{\mu,k}(t)^2\tau^2.
\end{equation}
Consequently, the reduced update \(\Delta_k\approx\widetilde\Sigma_k\) has an explicit order-\(\tau^2\) error whenever the component mean varies on a slower time scale than the averaging window.
\end{proposition}
The proof is given in the \ref{app:proofs}.

The omitted matrix \(R_k\) is the weighted temporal covariance of the component centroid within the averaging window. The estimate \eqref{eq:covariance-remainder-bound} shows that this contribution is of order \(O(\tau^2)\) for bounded centroid velocities, but its magnitude relative to the within-component covariance may increase during strong cluster interactions. The reduced update \(\Delta_k\approx\widetilde\Sigma_k\) should therefore be interpreted as a timescale-separated approximation rather than as a uniformly negligible correction. An a posteriori evaluation of \(R_k\) for the two-dimensional benchmark is reported in Section~\ref{subsec:test2}.

The next proposition chooses the affine drift so that the component means and covariances follow the reduced time-averaged targets in \eqref{maximizers_glog2}--\eqref{maximizers_glog3}. The weights and means then coincide with the exact maximizers of the averaged log-likelihood, while the covariance differs from its exact maximizer by the positive-semidefinite temporal-centroid term $R_k$ quantified in Proposition \ref{prop:log_lik2}.
\begin{proposition}\label{prop:coeffmixt_kernel}
Let \((m_k,\alpha_k)\), \(k=1,\dots,K\), solve \eqref{emMFG}. Assume that \(\widetilde\mu_k\) and \(\widetilde\Sigma_k\) are continuously differentiable, \(\widetilde\Sigma_k(t)\in\mathbb{S}_{++}^d\), and
\[
2\varepsilon I-\widetilde\Sigma_k'(t)\in\mathbb{S}_{++}^d.
\]
Let \(V_k(t)\) be the unique symmetric positive-definite solution of
\[
\widetilde\Sigma_k(t)V_k(t)+V_k(t)\widetilde\Sigma_k(t)
=2\varepsilon I-\widetilde\Sigma_k'(t),
\]
and define
\[
M_k(t)=\widetilde\mu_k(t)+V_k(t)^{-1}\widetilde\mu_k'(t).
\]
Assume moreover that $\nu_k(0)=\widetilde\mu_k(0)$, $T_k(0)=\widetilde\Sigma_k(0)$.
Then the mean and covariance generated by the Fokker--Planck equation satisfy
\[
\nu_k(t)=\widetilde\mu_k(t),
\qquad
T_k(t)=\widetilde\Sigma_k(t).
\]
Thus the model reproduces exactly the time-averaged weight and mean maximizers and the reduced covariance target \eqref{maximizers_glog3}; its discrepancy from the exact covariance maximizer is the positive-semidefinite matrix \(R_k\) bounded in \eqref{eq:covariance-remainder-bound}. Moreover, using the coefficients \(\widetilde\alpha_k\) defined in \eqref{maximizers_glog1},
\[
\mathbb{E}_{m(\cdot,t)}[X]=(g*\mathbb E_f[X])(t),
\]
and
\[
\begin{aligned}
\operatorname{Cov}_{m(\cdot,t)}[X]
&=(g*\mathbb E_f[XX^\top])(t)
 -(g*\mathbb E_f[X])(t)(g*\mathbb E_f[X])(t)^\top\\
&\quad +\sum_{k=1}^{K}\widetilde\alpha_k(t)\widetilde\mu_k(t)\widetilde\mu_k(t)^\top
 -\left(g*\sum_{k=1}^{K}\alpha_k\mu_k\mu_k^\top\right)(t).
\end{aligned}
\]
\end{proposition}
The proof is given in the \ref{app:proofs}.

The instantaneous model is sensitive to noise because it differentiates \(\mu_k\) and \(\Sigma_k\) numerically. The time-averaged model reduces this sensitivity through convolution. The asymmetric kernel \eqref{eq:drift-asym} is causal and can be used sequentially from past observations; the symmetric kernel \eqref{eq:sym-relax} is non-causal and is intended for offline reconstruction of a fully observed sequence.


\section{Numerical experiments and comparative assessment}
\label{sec:num_examples}

This section investigates the numerical behaviour of the proposed evolutionary clustering formulations and clarifies the different roles played by statistical estimation, temporal regularization, and controlled population
dynamics. The experiments have been reorganized around three complementary objectives.

The first test is a one-dimensional synthetic problem for which the exact evolution of the underlying groups is known. It is used to compare all the approaches considered in the paper in a very simplified, easy to visualise case.
The second test is a two-dimensional benchmark designed to compare the proposed formulations with simpler alternatives from the evolutionary clustering literature. In particular, the regularized models are compared
with a classical snapshot expectation-maximization procedure and with the same technique applied to temporally smoothed data, as well as with the established Evolutionary \(k\)-means method of Chakrabarti et al.~\cite{chakrabarti2006}. The latter provides an external non-EM evolutionary clustering baseline in which temporal coherence is introduced directly at the level of the cluster centroids. Accuracy, temporal regularity, and computational cost are reported and discussed in a common table.

The third test employs real time-dependent data. Since the example is more illustrative, the comparison focuses on the symmetric regularized model and the non-parametric mean field game formulation. The results are assessed in terms of data fidelity, temporal
coherence, robustness of the inferred populations, and computational time.

This organization distinguishes three levels of modeling. The static method, used as reference, contains no temporal information. The instantaneous and time-averaged models retain a Gaussian mixture representation and evolve its parameters through the finite-dimensional relations derived in Section~\ref{sec:EM_model}. In Test~2, Evolutionary \(k\)-means is included as an additional literature baseline: unlike the Gaussian-mixture formulations, it represents each cluster through a centroid and introduces historical consistency by regularizing the current centroids with respect to those obtained at the preceding time level. The non-parametric formulation instead evolves the component densities directly and computes their velocity fields from the coupled Hamilton--Jacobi--Fokker--Planck system. It therefore serves a different
purpose: it is more computationally demanding, but it does not require the shape of each component to be prescribed in advance.

\subsection{Test 1: one-dimensional moving and intersecting clusters}
\label{test1}

The first test is designed to compare all the clustering formulations
considered in this work under controlled conditions. In particular, it
compares independent static clustering, the instantaneous Gaussian model,
the asymmetric and symmetric temporally regularized models, and the
non-parametric mean field game formulation described above.

The computational domain and final time are
\[
    \Omega=[-0.5,1.5],
    \qquad
    T=5.
\]
The observed probability density is composed of three moving rectangular
profiles:
\[
    f(x,t)
    =
    \frac{1}{Z}
    \sum_{k=1}^{3}
    c_k\,\mathbbm{1}_{I_k(t)}(x),
    \label{eq:test1-density}
\]
where
\[
\begin{aligned}
 I_1(t)&=[0.12+0.10t,\;0.18+0.10t],\\
 I_2(t)&=[0.37+0.10t,\;0.42+0.10t],\\
 I_3(t)&=[0.60-0.10t,\;0.90-0.10t],
\end{aligned}
\]
with amplitudes
\[
    (c_1,c_2,c_3)=(3.0,7.0,1.5).
\]
The constant \(Z\) normalizes the total mass. The first two clusters move
towards the right, whereas the third one moves towards the left. The
supports therefore overlap during the simulation. In particular, the second
and third populations become difficult to distinguish near their
intersection.

This example tests two distinct difficulties. First, independent snapshot
clustering may exchange the component labels when the supports overlap.
Second, the rectangular data cannot be represented exactly by Gaussian
component densities. The experiment therefore provides a direct assessment
of the density-based formulation outside the Gaussian setting.

All methods are initialized from the same static clustering of the
observation at \(t=0\). For the static baseline, the clustering problem is
then solved independently at each time level. For the instantaneous model,
the moments are updated from the current responsibilities. The asymmetric
and symmetric methods use the temporal kernels introduced in
Section~\ref{sec:second_model} with parameter $\tau=0.5$. The asymmetric kernel is causal, whereas the
symmetric kernel uses both past and future observations.

For the non-parametric method, the interval is discretized using
\[
    N_x=301,
    \qquad
    \Delta x=\frac{2}{300}\simeq 6.67\times 10^{-3},
    \qquad
    \Delta t=\Delta x.
\]
The diffusion parameter and the coupling coefficient are $   \varepsilon=0.08$, $ \eta=0.30$. The positivity threshold is $ \delta_{\mathrm{num}}=10^{-10}$,
and the pseudo-time step in the Hamilton--Jacobi iteration is chosen as $   \Delta r=\Delta x$,
and the stopping tolerance is $ \mathrm{tol}_{\mathrm{HJ}}=10^{-7}$,
with at most \(500\) pseudo-time iterations at each physical time and for
each component. The value function computed at the preceding physical time
is used as the initial guess for the next stationary problem. This warm
start substantially reduces the number of pseudo-time iterations because
the coupling changes gradually in time.

The same diffusion coefficient is used in the Hamilton--Jacobi and
Fokker--Planck equations. After each Fokker--Planck update, negative values
arising at the level of round-off are truncated and the density is
renormalized. These corrections were negligible in the reported
computations.

\paragraph{Qualitative comparison}

Figure~\ref{fig:test1-snapshots} compares the reconstructed mixtures at a
selection of observation times. The static baseline gives an accurate
description of individual snapshots when the three supports are clearly
separated. However, since each frame is processed independently, component
labels and shapes may change abruptly near the intersections. This behaviour
is not necessarily visible in the total reconstructed density but becomes
apparent when the individual populations and their centroids are followed
over time.

The instantaneous model preserves the population evolution through the
Fokker--Planck dynamics and avoids repeatedly solving a complete static
fixed-point problem. It reacts rapidly to changes in the observed density,
but its moment estimates remain sensitive to overlap. Consequently, it can
display short-time oscillations or rapid exchanges between nearby mixture
configurations.

The asymmetric model reduces these oscillations by averaging information
over past observations. Its causal structure makes it appropriate for online
applications, but the use of past information introduces a systematic
tracking delay. The effect is particularly visible for the narrow moving
components, for which a small spatial displacement represents a substantial
fraction of their width.

The symmetric regularization uses information on both sides of the current
time and therefore produces smoother trajectories without the same causal
delay. It gives the most regular Gaussian reconstruction in this experiment,
at the price of requiring the entire observation interval and a global
offline computation.

The non-parametric computation uses the coupling
\eqref{eq_KL}. All populations are affected by the same
mixture-level Kullback--Leibler discrepancy, scaled by their current
mixture weights, while the term \(\log m_k\) introduces a
component-dependent contribution.

When \(\eta=0\), the component-dependent entropic term disappears.
Consequently, the populations are driven only by the discrepancy between
the complete reconstructed mixture and the observed density. Near an
intersection, this may produce similar velocity fields for different
components and may lead to a loss of component separation.

The choice \(\eta=0.30\) assigns weight \(0.30\) to the
component-dependent entropic contribution and weight \(0.70\) to the
mixture-level fidelity term. The fidelity term therefore remains dominant,
while the entropic contribution discourages excessive concentration and
helps preserve differentiated component profiles. This regularization is
spatial and component-dependent; it does not amount to a temporal
averaging of the observations.

A brief sensitivity analysis was performed by varying the coupling parameter
over \(\eta\in\{0.15,0.30,0.45\}\), while keeping all other numerical
parameters unchanged. The intermediate value \(\eta=0.30\) provides the best
performance for all the considered indicators, with mean and maximum centre
errors equal to \(0.0934\) and \(0.1617\), mean Wasserstein discrepancy
\(0.0338\), and trajectory variation \(E_{\mathrm{TV}}=1.6777\).
Both decreasing \(\eta\) to \(0.15\) and increasing it to \(0.45\) lead to
larger localization and density errors as well as less regular centroid
trajectories. These results indicate that the selected value
\(\eta=0.30\) provides a favourable balance between the
component-dependent entropic regularization and the mixture-level fidelity
term, while also showing that the model exhibits a non-negligible
sensitivity to this coupling parameter.

The computed populations are not constrained to be Gaussian. They can
therefore adapt their support and asymmetry directly through the
Fokker--Planck dynamics. This is the main distinction between the
non-parametric result and the instantaneous and time-averaged Gaussian
models. The increased flexibility comes at a higher computational cost,
since one nonlinear stationary Hamilton--Jacobi problem and one linear
Fokker--Planck problem must be solved for every component and every time
level.

\begin{figure}[H]
    \centering
    \includegraphics[width=12cm]{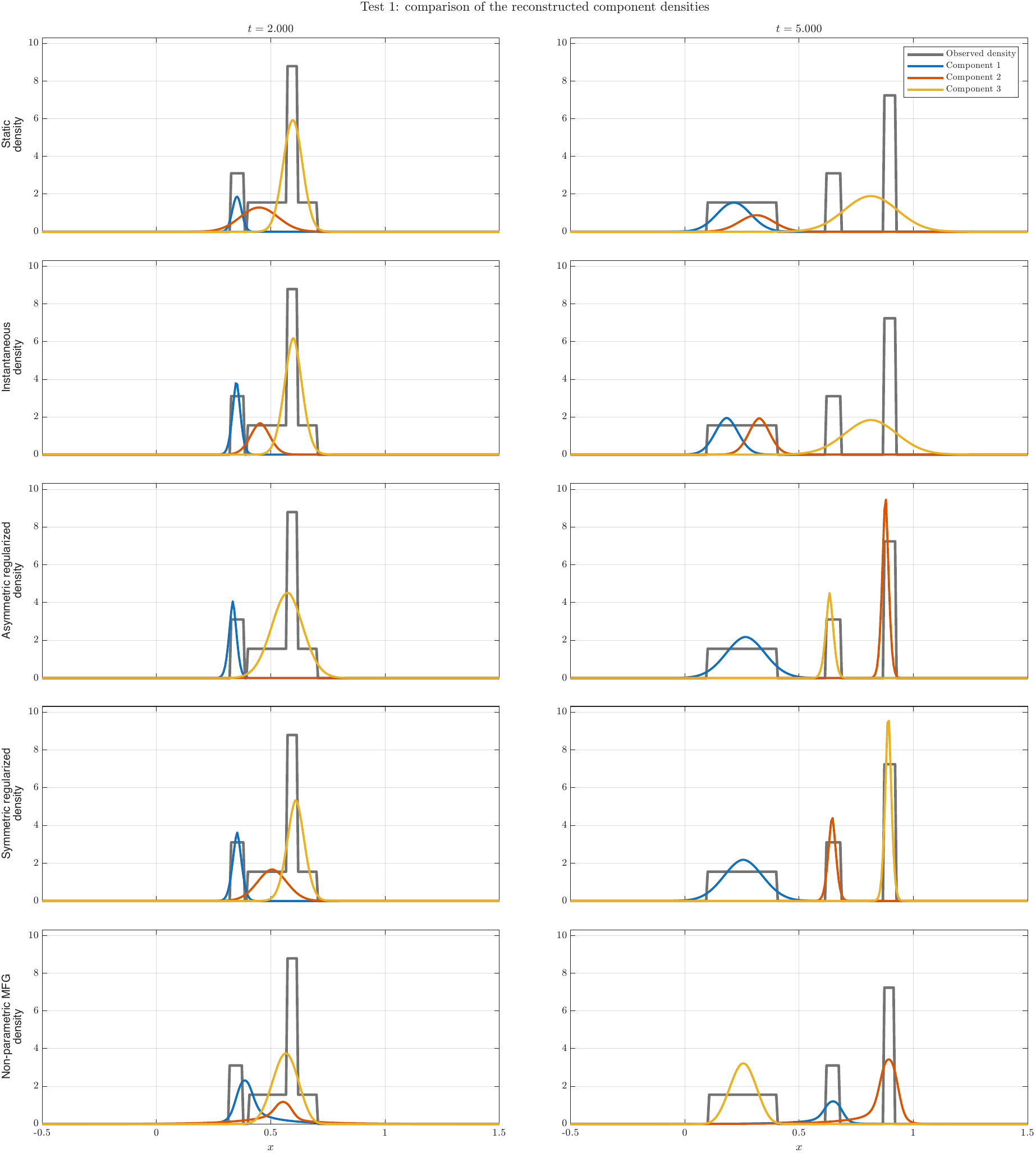}
    \caption{Test~1: reconstruction of the moving one-dimensional density at
    selected times. From top to bottom: static snapshot clustering,
    instantaneous evolution, asymmetric regularization, symmetric
    regularization, and the non-parametric mean field game model. The grey
    curve denotes the observed density and the coloured curves are the
    weighted component densities \(\alpha_k(t)m_k(\cdot,t)\).}
    \label{fig:test1-snapshots}
\end{figure}

\paragraph{Evolution of the population means.}

For every method, we monitor the mean position
\[
    \nu_k(t)
    =
    \int_\Omega x\,m_k(x,t)\,dx,
    \label{eq:test1-population-mean}
\]
where the populations are normalized to unit mass. To compare the inferred
means with the exact structures, the true centres are
\[
\begin{aligned}
 \nu_1^\star(t)&=0.15+0.10t,\\
 \nu_2^\star(t)&=0.395+0.10t,\\
 \nu_3^\star(t)&=0.75-0.10t.
\end{aligned}
\]
Figure~\ref{fig:test1-centroids} reports the corresponding trajectories.
The static and instantaneous methods react promptly but may undergo
oscillations or label exchanges near the intersections. The asymmetric
regularization produces smoother curves with a visible temporal delay. The
symmetric model removes most of the high-frequency variations. The non-parametric model is able to retain non-Gaussian component shapes
and the responsibility-based coupling improves the persistence
of the inferred assignments during overlap.

\begin{figure}[ht]
    \centering
    \includegraphics[width=\textwidth]{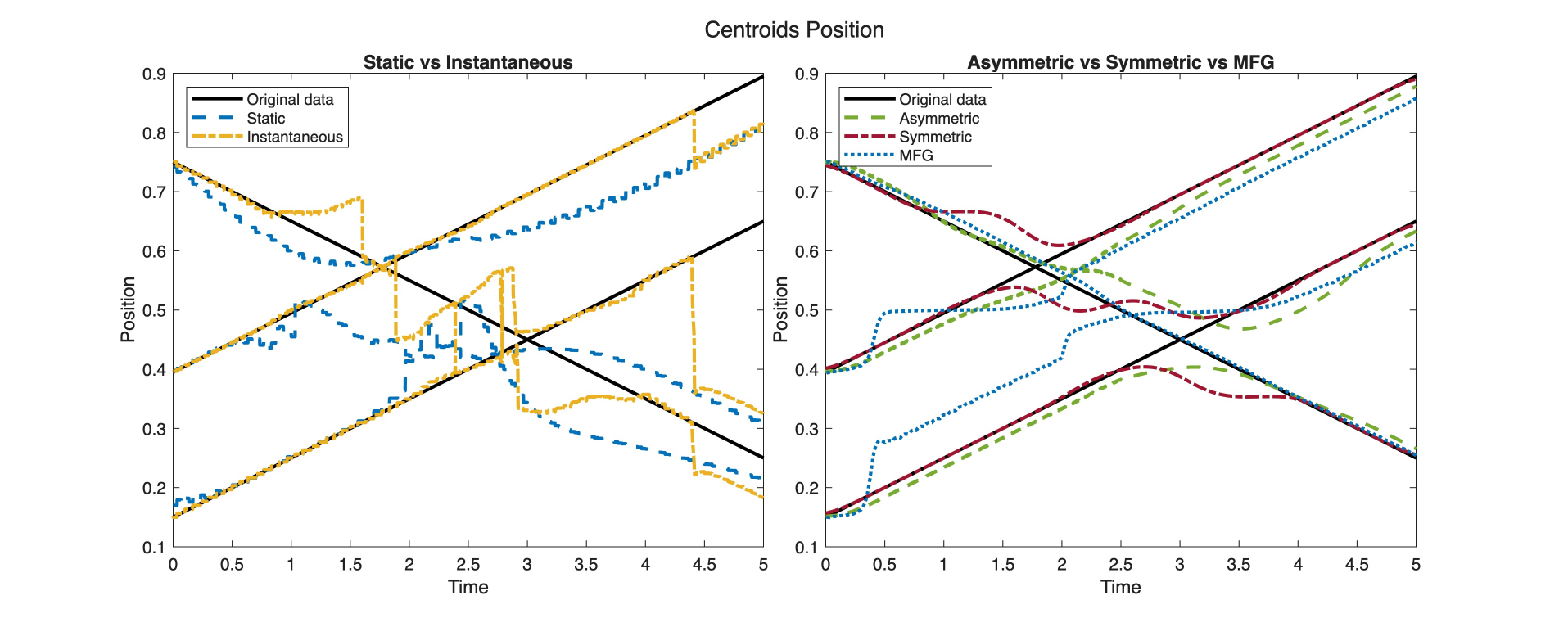}
    \caption{Test~1: evolution of the population means
    \(\nu_k(t)\). Dashed curves indicate the exact centres of the three
    rectangular profiles. The figure compares the static, instantaneous,
    asymmetric, symmetric, and non-parametric mean field game
    reconstructions.}
    \label{fig:test1-centroids}
\end{figure}

\paragraph{Quantitative indicators.}

We assess the methods using complementary measures of snapshot accuracy and
temporal coherence. The error in the population means is defined after
matching the reconstructed and exact components by the permutation
\(\pi_n\) minimizing their total distance:
\[
    E_\nu(t_n)
    =
    \min_{\pi\in\mathfrak S_K}
    \sum_{k=1}^{K}
    \left|
        \nu_k(t_n)-\nu_{\pi(k)}^\star(t_n)
    \right|.
    \label{eq:centroid-error-test1}
\]
This matching prevents a purely numerical change of component labels from
being interpreted as a localization error.

The discrepancy between the observed density and the complete reconstructed
mixture
\[
    m(x,t_n)
    =
    \sum_{k=1}^{K}
    \alpha_k^n m_k(x,t_n)
\]
is measured by the one-dimensional \(1\)-Wasserstein distance. On a uniform
grid, it is efficiently evaluated from the cumulative distributions:
\[
    W_1(f^n,m^n)
    =\Delta x
    \sum_{i=1}^{N_x}
    \left|
       \sum_{r=1}^{i} f_r^n \Delta x
       -
       \sum_{r=1}^{i} m_r^n \Delta x
    \right|.
    \label{eq:wasserstein-1d-test1}
\]
This formula is equivalent to the optimal-transport definition in one space
dimension and avoids solving a full linear programming problem at every time
step.

To measure temporal stability, we additionally introduce
\[
    E_{\mathrm{TV}}
    =
    \sum_{n=1}^{N_T}
    \sum_{k=1}^{K}
    \left|
       \nu_k(t_n)-\nu_k(t_{n-1})
    \right|,
    \label{eq:trajectory-variation-test1}
\]
after consistently matching the population labels. A small value indicates
a regular trajectory, although it should be interpreted together with
\(E_\nu\), since excessive smoothing can produce a small variation while
delaying the reconstructed clusters.

\begin{figure}[ht]
    \centering
    \includegraphics[width=\textwidth]{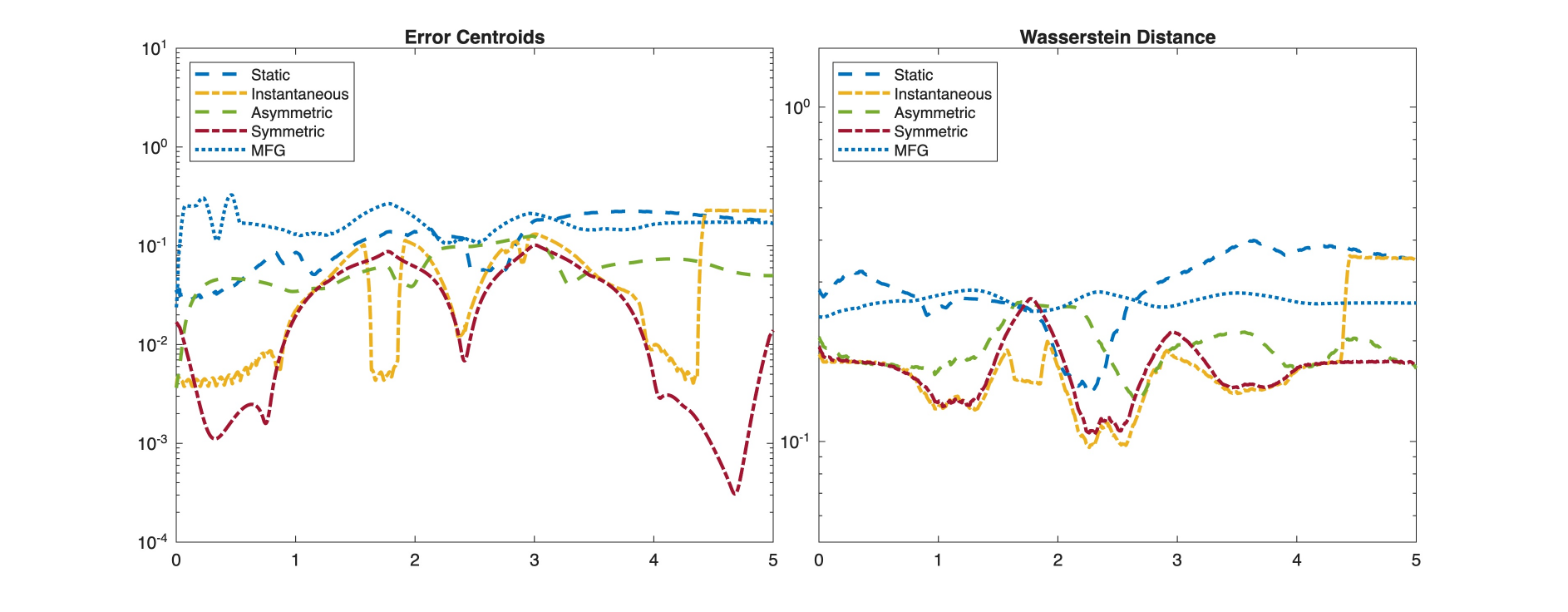}
    \caption{Test~1: quantitative comparison of the five methods.
    The panels report the population-mean error
    \(E_\nu(t)\), the \(1\)-Wasserstein discrepancy
    \(W_1(f(t),m(t))\), and, when displayed cumulatively, the temporal
    variation of the reconstructed trajectories.}
    \label{fig:test1-errors}
\end{figure}

For completeness, Table~\ref{tab:test1-summary} summarizes the errors and
computational costs over the complete time interval. All entries in a given
column are reported with the same number of decimal places. The reported
wall-clock times are measured on the same machine and include all iterations required
by each method.

\begin{table}[ht]
\centering
\caption{Test~1: accuracy, temporal regularity, and computational cost.
The mean and maximum errors are evaluated over all observation times.}
\vspace{0.5cm}
\label{tab:test1-summary}
\begin{tabular}{lrrrrr}
\toprule
Method
& Mean \(E_\nu\)
& Max. \(E_\nu\)
& Mean \(W_1\)
& \(E_{\mathrm{TV}}\)
& Wall time [s]\\
\midrule
Static & 0.1657 & 0.4529 & 0.0240 & 3.9526 & 32.00 \\
Instantaneous & 0.0873 & 0.4725 & 0.0126 & 3.5761 & 13.00 \\
Asymmetric regularized & 0.0626 & 0.1270 & 0.0208 & 2.1016 & 18.00 \\
Symmetric regularized & 0.0324 & 0.1042 & 0.0110 & 1.6127 & 64.00 \\
Non-parametric MFG & 0.0934 & 0.1617 & 0.0338 & 1.6777 & 53.00 \\
\bottomrule
\end{tabular}
\end{table}

The symmetric regularization gives the smallest mean and maximum
population-mean errors and the smallest mean Wasserstein discrepancy. It
also produces the most regular component trajectories, although it is the
most expensive method in this experiment because it requires a global
fixed-point iteration over the complete time interval.

The instantaneous model achieves a small mean Wasserstein error and has the
lowest computational cost, but its relatively large maximum centre error
confirms its sensitivity near component intersections. The asymmetric
method substantially reduces the centre error, although the causal temporal
average introduces a visible tracking delay.

The non-parametric MFG method produces a comparatively regular evolution
and a moderate maximum centre error. Its larger Wasserstein discrepancy is
partly explained by the diffusion and by the fact that the coupling acts on
the complete reconstructed mixture rather than independently fitting every
snapshot. Its main advantage is therefore not the instantaneous
reconstruction error, but the ability to evolve unrestricted component
densities without imposing Gaussian profiles.

\subsection{Test 2: comparison with parametric Gaussian-mixture and evolutionary clustering approaches}
\label{subsec:test2}

The purpose of this experiment is to compare the proposed evolutionary
clustering techniques with standard parametric approaches based on Gaussian
mixture models. In particular, we consider both an instantaneous
Expectation--Maximization method and an EM procedure applied to temporally
smoothed observations. The latter represents a natural way of introducing
temporal regularity into a classical parametric clustering pipeline without
modifying the EM algorithm itself.
In addition, to provide a comparison with an established
evolutionary clustering method that is not based on EM, we include the
Evolutionary \(k\)-means algorithm of Chakrabarti et al.~\cite{chakrabarti2006}.

The computational domain is
\[
    \Omega=[-0.5,1.5]^2,
\]
and the observation interval is
\[
    [0,T], \qquad T=1.5.
\]
The artificial density is generated by three moving disks having different
sizes, amplitudes, and trajectories. More precisely, their centres evolve as
\[
    c_1(t)=
    \begin{pmatrix}
        0.5\\
        0.5-0.25t
    \end{pmatrix},
    \qquad
    c_2(t)=
    \begin{pmatrix}
        0.75t\\
        1-0.30t
    \end{pmatrix},
    \qquad
    c_3(t)=
    \begin{pmatrix}
        0.75t\\
        0.75t
    \end{pmatrix}.
\]
The amplitudes and radii are
\[
    (a_1,a_2,a_3)=(1,2,0.5),
    \qquad
    (r_1^2,r_2^2,r_3^2)=(0.15,0.05,0.05).
\]
The three populations have different spatial scales and intensities, so that
the test simultaneously assesses the ability of the methods to track
components with unequal masses, preserve weak populations, and distinguish
clusters during partial overlap.

The reference density can be written as
\[
    f(x,t)
    =
    \frac{1}{Z}
    \sum_{k=1}^{3}
   a_k
    \mathbf{1}_{B(c_k(t),r_k)}(x),
\]
where \(a_k\) and \(r_k\) denote, respectively, the amplitude and size of the
\(k\)-th moving disk, while \(Z\) is a normalization constant. The numerical
discretization uses a triangular mesh with characteristic size
\(\Delta x = 8\times 10^{-3}\), and the time step is
\(\Delta t=\frac{\Delta x}{2}\).

The following clustering strategies are compared.

\begin{enumerate}
    \item \textbf{Independent snapshot EM.}
    At each observation time, a Gaussian mixture is fitted independently to
    the current density (see \cite{Bishop}). No information from the previous or subsequent time
    levels is used.

    \item \textbf{EM on temporally smoothed data.}
    The observation is first regularized in time,
    \[
        f_\sigma(x,t)
        =
        \int_0^T
        g_\sigma(t-s)f(x,s)\,ds,
    \]
    where \(g_\sigma\) is a Gaussian temporal kernel. A standard Gaussian
    mixture is then fitted independently to \(f_\sigma(\cdot,t)\).

    \item \textbf{Evolutionary \(k\)-means.}
    As an established discrete-time evolutionary clustering baseline, we
    consider the method introduced in~\cite{chakrabarti2006}. For the present
    continuous-density test, the method is implemented through a
    quadrature-weighted Euclidean adaptation: the barycentre \(x_i\) of each
    mesh triangle \(T_i\) is treated as a data point with weight
    \[
        w_i^n = |T_i|\,f(x_i,t_n).
    \]
    At every time level, a weighted \(k\)-means assignment is computed and the
    current centroid is regularized towards the closest centroid from the
    preceding time level. Denoting by \(\bar x_k^n\) and \(M_k^n\) the weighted
    mean and mass of the current \(k\)-th cluster, and by \(\ell(k)\) the
    closest previous centroid, the update has the form
    \[
        \widehat c_k^n
        =
        \frac{
        (1-c_p)M_k^n\bar x_k^n
        +
        c_p M_{\ell(k)}^{n-1}\widehat c_{\ell(k)}^{\,n-1}}
        {(1-c_p)M_k^n+c_pM_{\ell(k)}^{n-1}},
    \]
    and is iterated together with the nearest-centroid assignments until
    convergence. The change parameter is fixed at \(c_p=0.25\). This baseline
    provides cluster centroids and a hard spatial partition, but does not
    directly reconstruct a probability density.

    \item \textbf{Asymmetric evolutionary clustering.}
    The current clustering is regularized using information from previous
    observation times only as in Section~\ref{sec:second_model} (asymmetric regularization). This provides a causal evolutionary method.

    \item \textbf{Symmetric evolutionary clustering.}
    Temporal information is incorporated from both past and future
    observations as in Section~\ref{sec:second_model} (symmetric regularization). This non-causal formulation is particularly suitable for
    the offline analysis of complete time sequences.

    \item \textbf{Non-parametric MFG clustering.}
    Each component is represented by an evolving probability density rather
    than by a Gaussian distribution. The component densities satisfy a
    coupled Hamilton--Jacobi--Fokker--Planck system (see Section~\ref{sec:MFG}), with
    responsibility-based fidelity terms and periodic boundary conditions.
\end{enumerate}

The independent snapshot EM, the EM method applied to smoothed observations,
and the asymmetric and symmetric formulations are parametric Gaussian-mixture
approaches, since each component is described by a weight, a mean, and a
covariance matrix. Evolutionary \(k\)-means is instead a centroid-based
partitioning method with an explicit historical regularization, while the MFG
formulation is fully density-based and non-parametric. Thus, the comparison
includes both probabilistic Gaussian-mixture baselines and an established
non-EM evolutionary clustering method.

Recent stream-clustering methods discussed in the Introduction, such as
StreamDD and STM-Stream, are not included in the quantitative comparison because
they address a different observation model, namely sequential point streams with
concept drift and limited-memory processing. Evolutionary \(k\)-means is selected
here as the closest established history-regularized baseline that admits a
transparent adaptation to the present fixed-\(K\), spatial-density setting.

Figure~\ref{fig:test2-comparison} reports the original and temporally smoothed
data together with the outputs of all the clustering methods at four
representative times. For the density-based and Gaussian-mixture approaches,
the reconstructed total densities are displayed. Since Evolutionary
\(k\)-means does not provide a density reconstruction, its row displays the
three centroids and their Euclidean Voronoi partition.
For the Gaussian evolutionary computations in this test, the diffusion
coefficient is \(\varepsilon=1\) and the temporal averaging window is
\(\tau=0.5\). The EM baseline applied to temporally smoothed observations
uses \(\sigma=0.02\). The remaining numerical parameters are specified
below or are unchanged from the corresponding formulations of Test~1.

\begin{figure}[htbp]
    \centering
    \includegraphics[width=\textwidth]{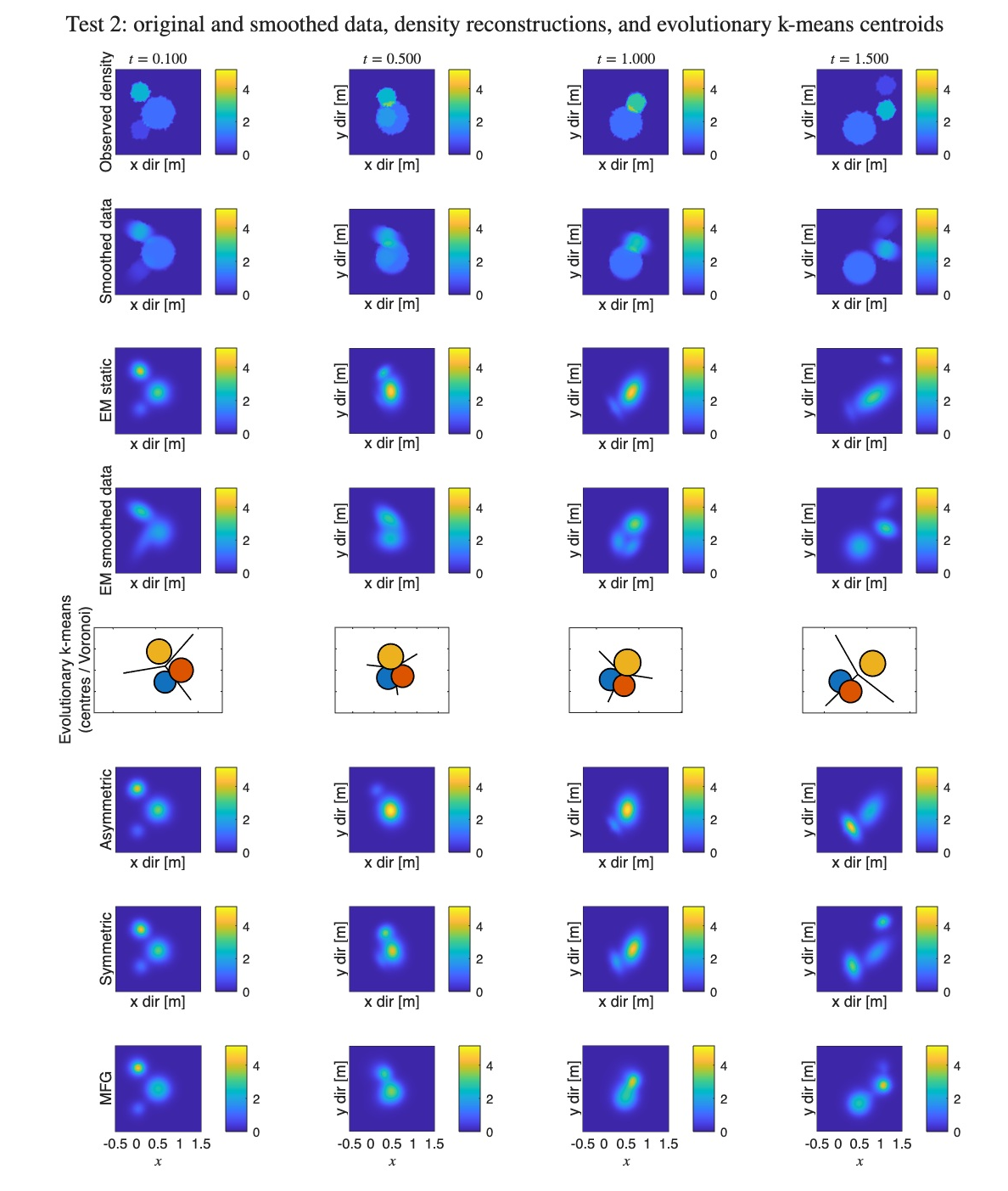}
    \caption{
    Test 2: comparison between the observed density, the temporally smoothed
    data, standard Gaussian-mixture approaches, Evolutionary \(k\)-means,
    and the proposed evolutionary methods. The columns correspond to
    \(t=0.1\), \(t=0.5\), \(t=1.0\), and \(t=1.5\).
    From top to bottom: original data, temporally smoothed data, independent
    snapshot EM, EM applied to smoothed observations, Evolutionary
    \(k\)-means, asymmetric evolutionary clustering, symmetric evolutionary
    clustering, and non-parametric MFG clustering. The Evolutionary
    \(k\)-means row reports centroids and Voronoi cells rather than a density
    reconstruction; the common colour scale therefore applies only to the
    density panels.
    }
    \label{fig:test2-comparison}
\end{figure}

The instantaneous EM reconstruction identifies the main regions of high
density but produces Gaussian components that are substantially more
concentrated than the original disk-shaped populations. Moreover, because the
clustering is recomputed independently at each time, the method does not
explicitly enforce consistency of component identities along the evolution.

Applying EM to the temporally smoothed observations regularizes the input
sequence before clustering and can reduce local frame-to-frame fluctuations. However, the preliminary
smoothing also modifies the spatial distribution of the observations. In
particular, nearby components may become more difficult to distinguish when
the temporal convolution causes their supports to overlap. As shown by the
quantitative indicators below, this preprocessing does not necessarily reduce
the global variation of the inferred centroid trajectories.

Evolutionary \(k\)-means tracks the three moving concentrations reasonably well
while they remain separated. Its historical term is designed to promote
consistency with the preceding clustering; nevertheless, in the present
experiment this does not translate into a smaller global trajectory variation.
A more difficult regime occurs when the second and third exact centres intersect. Indeed,
\(c_2(t)=c_3(t)\) at
\[
    t_\ast=\frac{1}{0.30+0.75}
    =\frac{20}{21}\simeq0.9524.
\]
Near and after this overlap, the hard \(k\)-means distortion criterion does
not enforce a one-to-one correspondence between centroids and physical
density modes. Moreover, in the evolutionary update two current centroids may
select the same closest predecessor. Consequently, a centroid may be
reallocated to a higher-mass region and one of the smaller-support
populations may subsequently be represented inaccurately. This behaviour
illustrates that historical consistency alone does not guarantee preservation
of component identity through a complete overlap.

The asymmetric evolutionary method provides greater temporal continuity than
instantaneous EM, but it may react with a delay to changes in the direction or
relative position of the populations. This effect is a consequence of its
causal formulation, since only past information is available when computing
the clustering at the current time.

The symmetric method gives a more regular reconstruction over the complete
time interval. The use of observations on both sides of the current time
reduces the delay associated with one-sided temporal regularization and helps
preserve the identity of the smaller components during close interactions.

The MFG method also produces temporally coherent clusters and preserves the
three-component structure throughout the simulation. Unlike the Gaussian
mixture approaches, its component densities are not restricted to an
elliptical parametric family. Consequently, the reconstructed populations may
adapt their shapes to the evolving observations.

\subsubsection{Quantitative comparison}

The visual comparison is complemented by the following indicators. Let
\(\widehat c_k^n\) be the reconstructed centre of the \(k\)-th population at
time \(t_n\), and let \(c_k(t_n)\) denote the corresponding exact centre.
Since the component labels may be permuted, at each time level we determine
the permutation \(\pi_n\) minimizing
\[
    E_c^n
    =
    \min_{\pi\in\mathcal S_3}
    \frac{1}{3}
    \sum_{k=1}^{3}
    \left\|
        \widehat c_k^n-c_{\pi(k)}(t_n)
    \right\|_2.
\]
We report both its time average
\[
    \overline E_c
    =
    \frac{1}{N_T}
    \sum_{n=1}^{N_T}E_c^n
\]
and its maximum value
\[
    E_c^{\max}
    =
    \max_{1\leq n\leq N_T}E_c^n.
\]

The accuracy of the reconstructed total mixture
\[
    \widehat f^n(x)
    =
    \sum_{k=1}^{3}
    \alpha_k^n m_k^n(x)
\]
is measured using the mean \(L^1\) error
\[
    \overline E_{L^1}
    =
    \frac{1}{N_T}
    \sum_{n=1}^{N_T}
    \int_\Omega
    \left|
        f(x,t_n)-\widehat f^n(x)
    \right|\,dx,
\]
and the mean Hellinger distance
\[
    \overline H
    =
    \frac{1}{N_T}
    \sum_{n=1}^{N_T}
    \left[
        \frac{1}{2}
        \int_\Omega
        \left(
            \sqrt{f(x,t_n)}
            -
            \sqrt{\widehat f^n(x)}
        \right)^2
        dx
    \right]^{1/2}.
\]

The \(L^1\) and Hellinger discrepancies are not reported for Evolutionary
\(k\)-means because the method returns centroids and a hard partition rather
than a probability-density reconstruction. For this baseline,
\(\overline E_c\), \(E_c^{\max}\), and \(E_{\mathrm{TV}}\) are computed
directly from the inferred centroid trajectories using the same permutation
criterion adopted for the other methods.

To quantify temporal regularity, we also consider the total variation of the
matched centre trajectories,
\[
    E_{\mathrm{TV}}
    =
    \sum_{n=1}^{N_T-1}
    \sum_{k=1}^{3}
    \left\|
        \widehat c_k^{n+1}
        -
        \widehat c_k^n
    \right\|_2.
\]
Finally, the wall-clock time is reported in order to compare the additional
cost associated with temporal regularization and with the solution of the
MFG system.

\begin{table}[htbp]
    \centering
    \caption{
    Quantitative comparison for Test 2. Lower values are better for all
    error and trajectory-variation indicators.
   A dash indicates a quantity that is not available for
    the corresponding method in the present comparison. In particular,
    \(L^1\) and Hellinger discrepancies are not defined natively for the
    centroid-based Evolutionary \(k\)-means baseline.
    }
    \vspace{0.5cm}
    \label{tab:test2-comparison}
    \begin{tabular}{lcccccc}
        \toprule
        Method
        & \(\overline E_c\)
        & \(E_c^{\max}\)
        & \(\overline E_{L^1}\)
        & \(\overline H\)
        & \(E_{\mathrm{TV}}\)
        & Wall time [s]
        \\
        \midrule
        Static EM
        & 0.1960 & 0.4934 & 0.6654 & 0.3544 & 11.5388 & 230.00 \\
        EM on smoothed data
        & 0.1246 & 0.2958 & 0.4956 & 0.3276 & 12.9685 &221.00\\
        Evolutionary \(k\)-means
        &0.2156
        & 0.4641
        & --
        & --
        &17.9285
        & 0.33\\
        Asymmetric evolutionary
        & 0.2467 & 0.3305 & 0.8293 & 0.4293 & 4.2375 & 320.00 \\
        Symmetric evolutionary
        & 0.1207 & 0.2063 & 0.6208 & 0.3405 & 3.6645 &612.00\\
        Non-parametric MFG
        & 0.1793 & 0.4259 & 0.6608 & 0.3627 & 3.2538 &651.94 \\
        \bottomrule
    \end{tabular}
\end{table}

Table~\ref{tab:test2-comparison} is intended to distinguish three different
properties of the algorithms. The centre errors assess their ability to track
the individual populations, the \(L^1\) and Hellinger errors evaluate the
quality of the reconstructed total density, while
\(E_{\mathrm{TV}}\) measures the temporal regularity of the identified
trajectories. These criteria are complementary: a method may accurately
approximate the total mixture while still exchanging component labels or
producing irregular centre trajectories.
For Evolutionary \(k\)-means, the density-based indicators
are not available, so its comparison is restricted to localization,
trajectory regularity, and computational cost. In particular, a small
trajectory variation would not by itself imply successful tracking if two
centroids have collapsed onto the same density mode.

The results reveal a clear distinction between density reconstruction and
component tracking. EM applied to smoothed observations gives the smallest
mean \(L^1\) and Hellinger discrepancies, showing that temporal
preprocessing improves the approximation of the total density. However, in the present experiment this preprocessing does not reduce the
total variation of the centroid trajectories relative to independent snapshot
EM, since \(E_{\mathrm{TV}}=12.9685\) compared with \(11.5388\) for the
static baseline. Its centre paths also remain considerably less regular than
those obtained with the proposed asymmetric, symmetric, and MFG formulations.

The Evolutionary \(k\)-means baseline provides reasonably accurate centroid localization
during the initial part of the evolution, when the three modes are spatially
distinct. Its performance deteriorates around and after the intersection of
the second and third populations at \(t_\ast\simeq0.9524\). This is consistent
with the hard-partition nature of the method: the historical penalty promotes
continuity of the clustering but does not enforce persistent one-to-one
component identities when two modes coincide. The resulting centroid
reallocations are also reflected in the relatively large trajectory variation,
\(E_{\mathrm{TV}}=17.9285\), the largest value among the methods considered in
this test. The comparison therefore
highlights a distinction between temporal regularization of a partition and
explicit tracking of evolving mixture components.

The symmetric evolutionary method gives the smallest mean and maximum
centre errors. The use of past and future information reduces causal delay
and improves the localization of the individual populations.

The non-parametric MFG method gives the smallest trajectory variation,
which is close to the variation of the exact centre paths. Its density and
centre errors are not the smallest, but it provides the most regular
non-parametric evolution without imposing Gaussian component shapes. This comes at the
price of a higher computational cost.

\paragraph{A posteriori assessment of the covariance approximation.}
We finally assess numerically both the reduced covariance approximation
introduced in Proposition~\ref{prop:log_lik2} and the positivity condition
required by the affine-drift construction. For the time-averaged models,
the exact covariance maximizer is
\[
    \Delta_k^\star(t)
    =
    \widetilde\Sigma_k(t)+R_k(t),
\]
where \(R_k\) represents the covariance generated by the motion of the
component centroid inside the temporal averaging window. In addition to its
spectral norm, we consider the relative correction
\[
    r_k(t)
    =
    \frac{\|R_k(t)\|_2}
         {\|\widetilde\Sigma_k(t)\|_2}.
\]
In Test~2, the asymmetric regularization gives
\[
    \max_{k,t}\|R_k(t)\|_2
    =
    1.3695\times10^{-3},
    \qquad
    \max_{k,t} r_k(t)
    =
    9.345\%,
\]
while the maximum relative correction in the overlap region
\(|t-t_\ast|\leq0.1\), with \(t_\ast=20/21\), is \(7.690\%\).
For the symmetric regularization,
\[
    \max_{k,t}\|R_k(t)\|_2
    =
    2.8902\times10^{-3},
    \qquad
    \max_{k,t} r_k(t)
    =
    36.491\%,
\]
and the latter maximum is attained in the overlap region. Thus,
\(R_k\) remains small in absolute norm for both regularizations, whereas its
relative contribution may become non-negligible for the symmetric model
during strong component interaction. This confirms that
\(\Delta_k\simeq\widetilde\Sigma_k\) should be interpreted as a
timescale-separated reduced approximation rather than as a uniformly
negligible correction.

We also evaluate the positivity condition entering the Sylvester equation,
using the same backward difference employed in the numerical scheme,
\[
    Q_k^n
    =
    2\varepsilon I
    -
    \frac{\Sigma_k^n-\Sigma_k^{n-1}}{\Delta t},
\]
with \(\Sigma_k\) replaced by \(\widetilde\Sigma_k\) for the time-averaged
models. For the reported two-dimensional Gaussian computations,
\(\varepsilon=1\). The sufficient thresholds
\[
    \varepsilon_{\min}
    =
    \frac12
    \max_{k,n}
    \lambda_{\max}(\dot\Sigma_k^n)
\]
are \(0.0701\) and \(0.0630\) for the asymmetric and symmetric
regularizations, respectively. Accordingly,
\(Q_k^n\succ0\) at every valid time level for both methods, with minimum
eigenvalues approximately \(1.860\) and \(1.874\), respectively.
Temporal averaging therefore regularizes not only the centroid trajectories
but also the covariance derivatives entering the affine drift.

The instantaneous model behaves differently. Owing to the much sharper
variations of the estimated covariance matrices near component
interactions, its corresponding threshold is approximately \(15.97\),
which exceeds the value \(\varepsilon=1\) used in the computation.
Consequently, the sufficient positivity condition is satisfied at
\(99.44\%\) of the valid discrete time levels but is violated at a small
number of isolated steps, with
\(\min_{k,n}\lambda_{\min}(Q_k^n)\simeq-29.94\).
Hence the hypotheses of Proposition~\ref{prop:equal_data} are verified
uniformly by the two temporally regularized computations but not by the
instantaneous run. This observation is consistent with the greater
sensitivity of the instantaneous covariance estimates to rapid variations
and provides an additional numerical motivation for temporal averaging.

\subsection{Test 3: seasonal clustering of east Australian
humpback-whale movements}
\label{subsec:test3-whales}

The third experiment investigates the application of the proposed
evolutionary clustering methods to a real animal-movement dataset. The data
were obtained from the Movebank Data Repository and are described by
Andrews-Goff et al.~\cite{AndrewsGoff2023Data,
AndrewsGoff2023Whales}. The dataset contains satellite-derived locations of
east Australian humpback whales, \emph{Megaptera novaeangliae}, belonging
to a unique breeding stock, noted E1. Satellite tags were deployed on 50 individuals
between 2008 and 2010 during northward and southward migrations and on
Antarctic feeding grounds.

The original data package contains both Argos-derived locations and
model-based position estimates. To avoid combining different
representations of the same tracks, we retain only visible observations
that are not marked as algorithmic outliers and are identified as modelled
positions. This preprocessing yields 13\,311 locations associated with
48 individuals.

The selected tracks cover a large geographical region, extending from the
tropical waters of eastern Australia to Antarctic feeding grounds. Some
trajectories also cross the international date line.

The study in~\cite{AndrewsGoff2023Whales} identifies a pronounced seasonal
organization of the movements. Tagged whales were generally located near
their northern limit in July and near their southern limit in March.
Northward migration was most rapid during May and June, whereas the largest
southward displacements occurred during November and December. Tagged
animals were located on Antarctic feeding grounds predominantly between
January and May and approached the northern breeding grounds between July
and August. This seasonal structure makes the dataset suitable for testing
dynamical soft-clustering methods.

\subsubsection{Construction of the seasonal densities}

The observations are pooled according to calendar month, independently of
the year of acquisition. We thus construct a month-of-year seasonal
sequence of twelve probability densities,
\[
    f^n:\Omega\longrightarrow\mathbb{R}_+,
    \qquad n=1,\ldots,12,
\]
where \(\Omega\subset\mathbb{R}^2\) is the projected spatial domain.

The resulting sequence is a pooled seasonal representation and not the
chronological trajectory of a fixed set of individuals. Observations
belonging to different years and different animals may contribute to
consecutive monthly densities. For the purpose of temporal regularization,
the sequence is treated as periodic, so that December and January are
regarded as consecutive states.

The number of available positions differs considerably between individuals.
Direct kernel-density estimation using all locations with equal point
weights would therefore give greater importance to animals whose tags
transmitted more frequently. To reduce this effect, every individual active
during a given month is assigned the same total mass.

Let \(\mathcal I_n\) denote the set of individuals observed during month
\(n\), and let \(R_{i,n}\) be the number of retained positions of individual
\(i\) during that month. The monthly density is defined by
\[
    f^n(x)
    =
    \frac{1}{|\mathcal I_n|}
    \sum_{i\in\mathcal I_n}
    \frac{1}{R_{i,n}}
    \sum_{r=1}^{R_{i,n}}
    \mathcal K_{\gamma_{\mathrm{KDE}}}
    \left(x-X_{i,n,r}\right),
    \label{eq:whale-monthly-density}
\]
where \(X_{i,n,r}\) is the projected position of the \(r\)-th observation
and
\[
    \mathcal K_{\gamma}(z)
    =
    \frac{1}{2\pi\gamma^2}
    \exp\left(
        -\frac{\|z\|_2^2}{2\gamma^2}
    \right).
\]
The bandwidth used in the present computation is
\[
    \gamma_{\mathrm{KDE}}=280\,\mathrm{km}.
\]
Each monthly field is subsequently normalized so that
\[
    \int_\Omega f^n(x)\,dx=1.
\]

This weighting corrects for differences in the number of fixes contributed
by each active animal within a month. It does not correct for variation in
deployment design, tag duration, or the number and composition of active
individuals. Consequently, \(f^n\) represents the normalized distribution
of the available tagged-whale locations and not the abundance density of
the complete breeding stock.

For months without usable observations, a density is obtained by circular
linear interpolation between the nearest observed monthly fields and is
then normalized to unit mass. Such fields are regarded as numerically
imputed observations and are distinguished from directly observed months
when evaluating data-fidelity indicators.

Because the trajectories cover a large geographical region and may cross
the international date line, latitude and longitude are not treated as
Cartesian coordinates. The observations are mapped through a Lambert
azimuthal equal-area projection centred at
\(42^\circ\mathrm{S}\), \(150^\circ\mathrm{E}\). The use of an equal-area
projection is particularly convenient for the construction and
normalization of spatial probability densities.

The projected domain is discretized on a \(54\times54\) grid. A common
characteristic length is then used to nondimensionalize both spatial
coordinates while preserving the aspect ratio of the projected domain.
The inverse projection is employed for the final georeferenced
visualizations.

We seek a decomposition
\[
    f^n(x)
    \simeq
    \widehat f^n(x)
    =
    \sum_{k=1}^{K}
    \alpha_k^n m_k^n(x),
    \qquad K=3,
    \label{eq:whale-mixture}
\]
where
\[
    \alpha_k^n\geq0,
    \qquad
    \sum_{k=1}^{K}\alpha_k^n=1,
    \qquad
    \int_\Omega m_k^n(x)\,dx=1.
\]

The number of components is fixed a priori to \(K=3\) to provide a
parsimonious representation of the dominant spatial modes. The densities
\(m_k^n\) are interpreted as mathematical modes of the pooled tagged-whale
locations. They may capture recurrent concentrations associated with
breeding areas, migratory corridors, or Antarctic feeding regions, but they
do not by themselves identify distinct biological populations,
behavioural states, or groups of individual animals.

\subsubsection{Results and interpretation}

Figure~\ref{fig:whale-seasonal-comparison} compares four representative
phases of the annual migration: November, January, April, and July. The first
row reports all retained state-space-model positions occurring in the
corresponding calendar month, pooled across the available years. The second
row shows the monthly kernel-density estimates. The third and fourth rows
contain the reconstructions obtained with the symmetric method and the
MFG method, respectively.

\begin{figure}[htbp]
    \centering
    \includegraphics[width=\textwidth]
    {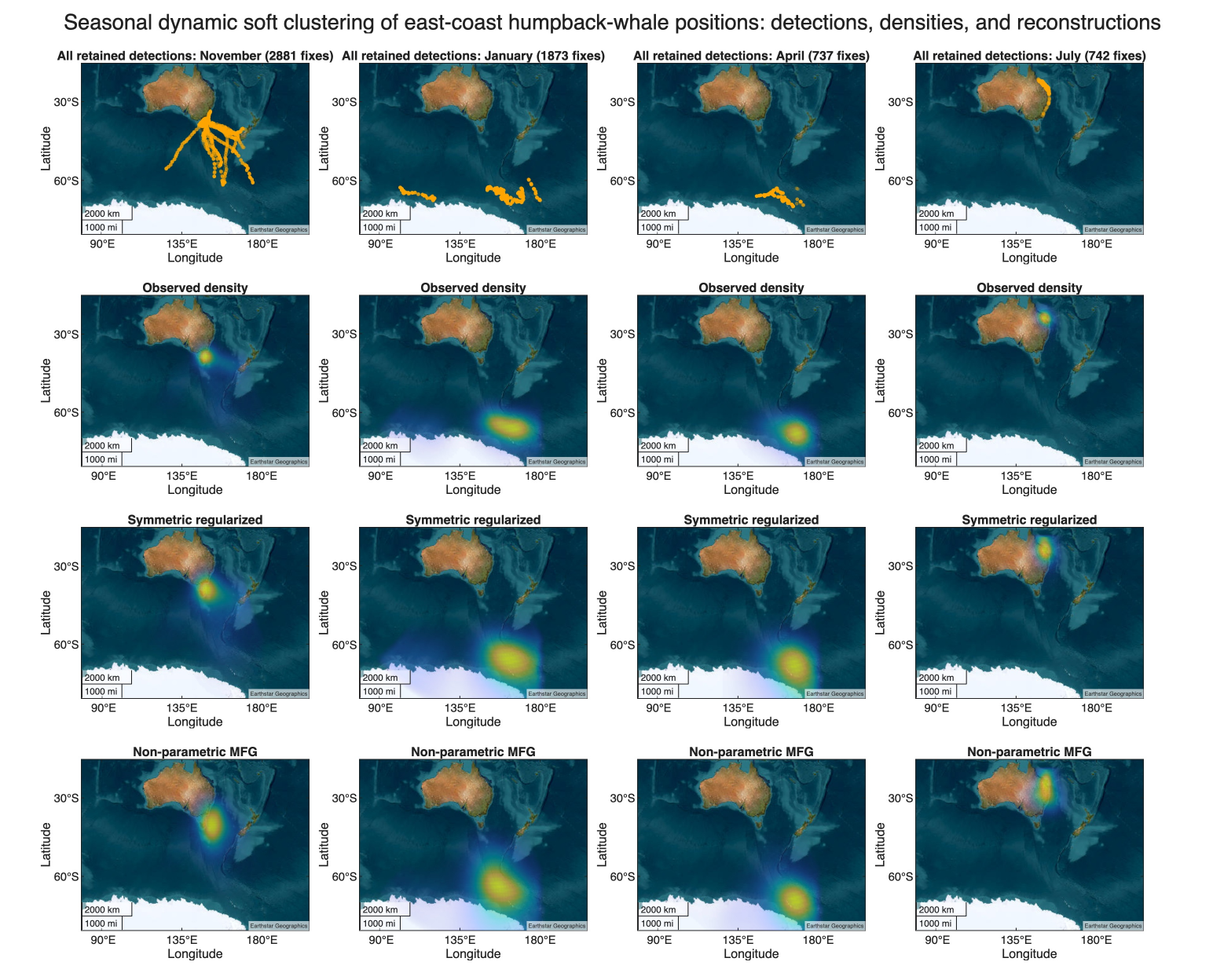}
    \caption{
    Seasonal dynamic soft clustering of east Australian humpback-whale
    movements. Columns correspond to November, January, April, and July.
    From top to bottom: retained state-space-model locations, observed
    monthly kernel-density estimate, symmetric regularized reconstruction,
    and non-parametric MFG reconstruction.
    All panels use the same geographical window,
    \(72^\circ\mathrm{S}\)--\(10^\circ\mathrm{S}\) and
    \(105^\circ\mathrm{E}\)--\(180^\circ\mathrm{E}\).
    Density overlays use continuous transparency, with zero density rendered
    completely transparent.
    }
    \label{fig:whale-seasonal-comparison}
\end{figure}

The raw observations exhibit a clear seasonal displacement. During the
austral summer, the principal density is concentrated in the Southern Ocean
and close to Antarctic feeding grounds. During autumn and early winter, the
distribution moves northward towards the eastern Australian coast. By July,
the dominant spatial mass is located near the northern part of the observed
range. This qualitative evolution agrees with the seasonal migration
described in~\cite{AndrewsGoff2023Whales}.

The symmetric method produces spatially regular components and benefits from
using information from both adjacent months. It is particularly effective at
suppressing irregular changes caused by months with few observations. Its
non-causal nature, however, means that the reconstructed component at a given
month is influenced by both earlier and later observations.

The MFG method provides a sequential evolution of
non-parametric component densities along the periodic monthly sequence. The entropic regularisation and continuity intrinsically contained in the model is important for preserving
component identity in presence of sparse data during the long displacements between Australia and
Antarctica. 

The inferred barycentre trajectories are shown in
Figure~\ref{fig:whale-centre-trajectories}. They provide a compact
representation of the annual motion of the three soft components, consistently with the seasonal pattern reported in \cite{AndrewsGoff2023Whales} although
they should not be interpreted as trajectories of individual whales.

\begin{figure}[htbp]
    \centering
    \includegraphics[width=\textwidth]
    {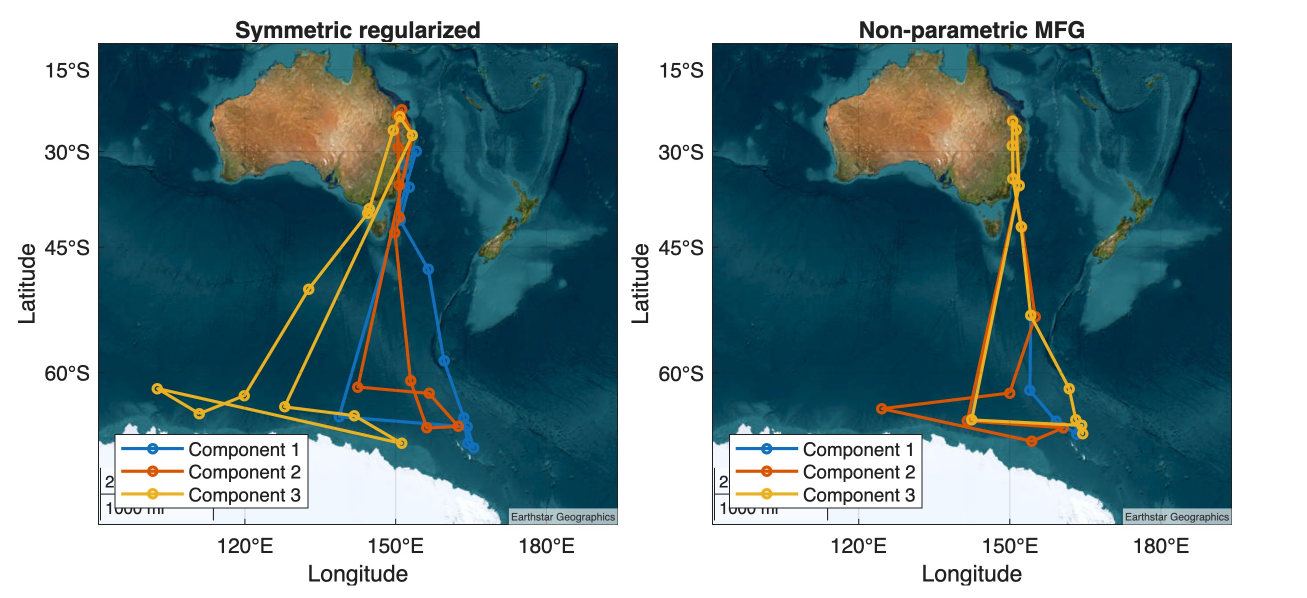}
    \caption{
    Annual trajectories of the component barycentres obtained with the
    symmetric regularized method and the MFG method.
    The plotted curves describe the motion of spatial density modes and not
    the paths of individual tagged animals.
    }
    \label{fig:whale-centre-trajectories}
\end{figure}

This experiment demonstrates that the proposed methods can be applied to
large-scale ecological movement data without assuming Gaussian component
shapes. The symmetric approach is well suited to retrospective seasonal analysis,
whereas the MFG formulation provides a sequential dynamical
interpretation of the reconstructed seasonal density modes. Nevertheless, the resulting clusters describe
statistical modes of the pooled observations. They should not be interpreted
as evidence of distinct biological subpopulations without additional
ecological, genetic, or behavioural information.

\section{Conclusions}
\label{sec:conclusions}

In this work, we have introduced a mean field game framework for
evolutionary soft clustering. The central idea is to represent each cluster
by a probability density whose evolution is governed by a controlled
Fokker--Planck equation, while the associated velocity field is determined
through a stationary Hamilton--Jacobi problem. The interaction between the
component densities, the complete reconstructed mixture, and the observed
data is encoded through a variational coupling. This construction provides
a continuous-time, mass-preserving, and control-theoretic interpretation of
evolutionary clustering.

A first contribution of the paper is the connection between this general
density-based formulation and classical Gaussian-mixture estimation. By
restricting the component densities to Gaussian profiles and selecting an
appropriate affine drift, we have shown that their means, covariance
matrices, and mixture weights reproduce the continuous-time counterparts of
the parameters obtained through the Expectation--Maximization procedure.
The resulting correspondence gives a dynamical interpretation of EM and
clarifies how finite-dimensional mixture estimation can be embedded into a
population-evolution model.

Within the Gaussian setting, we have also introduced asymmetric and
symmetric time-averaged log-likelihood functionals. The asymmetric
regularization uses past observations only and is therefore compatible with
causal or online processing, although it may introduce a tracking delay. The
symmetric formulation uses observations on both sides of the current time
and is more appropriate for retrospective reconstruction of a complete
sequence. These two alternatives distinguish temporal regularization of the
clustering variables from the simpler strategy of smoothing the observations
before applying an otherwise independent snapshot algorithm.

The numerical experiments confirm that these distinctions are relevant in
practice. In the one-dimensional test with moving and intersecting
rectangular profiles, no single method is uniformly preferable with respect
to all indicators. Snapshot and instantaneous approaches adapt rapidly to
the current observation but may exhibit oscillations or component-label
exchanges near intersections. Temporal regularization improves the
coherence of the reconstructed trajectories, with the symmetric formulation
providing the most accurate and regular Gaussian reconstruction in the
reported experiment. The non-parametric MFG method achieves a comparably
regular evolution while retaining component shapes that are not constrained
to remain Gaussian.

The two-dimensional benchmark further separates reconstruction accuracy
from component tracking. Applying EM to temporally smoothed observations
improves the approximation of the total density, but it does not enforce
continuity of the individual component identities.
The additional comparison with Evolutionary \(k\)-means
shows that historical regularization of centroid positions does not by
itself guarantee preservation of weak component identities during strong
overlap: in this test, the centroid-based method tracks the separated
populations reasonably well at early times but deteriorates near and after
the intersection of the two smaller components.
The symmetric evolutionary method gives the most accurate localization
of the moving populations, whereas the non-parametric MFG formulation
produces the smallest trajectory variation, close to that of the prescribed
exact paths.
The a posteriori covariance check further shows that the
temporal-centroid remainder \(R_k\) stays below \(10\%\) of the reduced
within-component covariance for the asymmetric regularization, whereas the
relative correction can reach about \(36.5\%\) for the symmetric model
during the strongest overlap. This confirms that the reduced covariance
update is a timescale-separated approximation and should not be interpreted
as uniformly exact in strongly interacting regimes.

The real-data experiment demonstrates the applicability of the proposed
ideas to a large-scale ecological movement dataset. Monthly probability
densities were constructed from modelled locations of east Australian
humpback whales by pooling observations across years and balancing the
contribution of the active tagged individuals. The symmetric grid-based
method and a MFG formulation both recover a pronounced
seasonal displacement between eastern Australia and the Southern Ocean.

The whale-data example also illustrates the care required when interpreting
the inferred clusters. The monthly fields combine observations from
different years and different tagged animals, and the resulting component
densities represent statistical spatial modes of the pooled locations. They
may be associated with recurrent concentration regions, migratory corridors,
or feeding and breeding areas, but they do not by themselves identify
biological subpopulations, behavioural states, or trajectories of individual
animals.

Taken together, the experiments show that the main advantage of the MFG
formulation is structural rather than uniformly superior snapshot accuracy.
The component densities can evolve without belonging to a prescribed
parametric family, mass is conserved by the Fokker--Planck dynamics, and the
velocity fields arise from an explicit optimization principle. This
flexibility comes at a substantially higher computational cost, because a
stationary nonlinear Hamilton--Jacobi problem and a Fokker--Planck update
must be solved for each component and time level.

Several questions remain open. From the analytical point of view, it would
be valuable to establish well-posedness and stability results for the
nonlinear quasi-stationary multi-population system and to characterize the
conditions under which different components remain identifiable. A
convergence and error analysis of the Hamilton--Jacobi--Fokker--Planck
discretization would also provide a firmer basis for the numerical
comparisons.

From the modelling perspective, future work should investigate systematic
criteria for choosing the number of components and the coupling parameters,
as well as the design of component-specific attraction and repulsion terms
that preserve non-parametric flexibility without requiring problem-dependent
tracking corrections. For irregularly sampled real data, uncertainty in the
observations, missing time intervals, unequal tag duration, and variation in
the sampled individuals should be incorporated directly into the model
rather than handled only through preprocessing.

Finally, the computational cost motivates the development of scalable
solvers, including adaptive spatial discretizations, parallel
component-wise algorithms, multilevel Hamilton--Jacobi solvers, and reduced
or particle-based approximations. A genuinely online version of the method,
based on a single causal pass and uncertainty-aware updating, would be
particularly relevant for streaming applications. These developments would
help turn the proposed framework from a general modelling principle into a
practical tool for high-dimensional and large-scale evolutionary clustering.

\section*{Authorship contribution statement}

Alessio Basti: Conceptualization, Methodology, Formal analysis, Software, Validation, Investigation, Data curation, Visualization, Writing -- original draft, Writing -- review \& editing. Fabio Camilli: Conceptualization, Methodology, Formal analysis, Software, Validation, Investigation, Data curation, Visualization, Writing -- original draft, Writing -- review \& editing. Adriano Festa: Conceptualization, Methodology, Formal analysis, Software, Validation, Investigation, Data curation, Visualization, Writing -- original draft, Writing -- review \& editing. All authors contributed equally to this work.

\section*{Declaration of competing interest}

The authors declare that they have no known competing financial interests or personal relationships that could have influenced the work reported in this paper.

\section*{Acknowledgments}

This work was partially supported by GNAMPA 2025 project \textit{“Clustering dinamico e giochi a campo medio”}.

\section*{Data availability}

The data used in the real-data experiment are publicly available from the Movebank Data Repository through the sources cited in the manuscript. The numerical data and code supporting the findings of this study are available from the corresponding author upon reasonable request.


\appendix
\section{Proofs}\label{app:proofs}
\begin{proof}[Proof of Proposition \ref{prop: explicit_sol}]

	Replacing \eqref{eq:drift} in the FP equation in \eqref{emMFG}, we   obtain 
	\begin{equation}\label{FP1}
		\pd_t m_k(x,t)-	\e    \Delta m_k(x,t)-\diver(m_k(x,t)V_k(t)(x-M_k(t)))=0.
	\end{equation}
	We look for a non-negative normalized (probability density function) solution of such   equation. Let us assume an ansatz of the form
	\[
	m_k(x,t)= c_k(t)\exp\Bigl(-\frac{1}{2}(x-\nu_k(t))^\top \,T_k(t)^{-1}(x-\nu_k(t))\Bigr),
	\]
	where $c_k(t)$ is a scalar, $\nu_k(t)\in \mathbb{R}^d$ is the mean, and $T_k(t)$ is a symmetric, positive definite $d\times d$ matrix. The normalization condition
	\[
	\int_{\mathbb{R}^d} m_k(x,t)\,dx = c_k(t)\sqrt{(2\pi)^d \det T_k(t)} = 1,
	\]
	leads to 
	\(
	c_k(t)= 1/(\sqrt{(2\pi)^d \det T_k(t)}).
	\)
	By defining
	\[
	g_k(x,t)= -\frac{1}{2}z(x,t)^\top\,T_k(t)^{-1}z(x,t)\quad \text{with}\quad z(x,t):=x-\nu_k(t),
	\]
	the ansatz can be written as
	\(
	m_k(x,t)= c_k(t)\,e^{\,g_k(x,t)}.
	\)
	Differentiating in time,
	\[
	\partial_t m_k(x,t)= c'_k(t)e^{\,g_k(x,t)} + c_k(t)e^{\,g_k(x,t)}\partial_t g_k(x,t)
	= m_k(x,t)\Biggl[\frac{c'_k(t)}{c_k(t)} + \partial_t g_k(x,t)\Biggr].
	\]
	Since $T_k(t)$ and $\nu_k(t)$ depend on $t$, we compute
	\[
	\partial_t g_k(x,t)= -\frac{1}{2}\partial_t\Bigl(z(x,t)^\top\,T_k(t)^{-1}\,z(x,t)\Bigr).
	\]
	Noting that $\partial_t z(x,t) = -\nu'_k(t)$ and using the identity
	\[
	\frac{d}{dt}\,T_k(t)^{-1} = -T_k(t)^{-1}\,T'_k(t)\,T_k(t)^{-1},
	\]
	we obtain
	\[
	\begin{split}
		\partial_t g_k(x,t)
		& = -\frac{1}{2}\Bigl[(\partial_t z(x,t))^\top\,T_k^{-1}\,z(x,t) + z(x,t)^T\,\frac{d}{dt}\bigl(T_k^{-1}\bigr)\,z(x,t) + z(x,t)^\top\,T_k^{-1}\,(\partial_t z(x,t))\Bigr]\\[1mm]
		& = -\frac{1}{2}\Bigl[(-\nu'_k(t))^\top\,T_k^{-1}\,z(x,t) + z(x,t)^\top\,\bigl(-T_k^{-1}T'_kT_k^{-1}\bigr)\,z(x,t) + z(x,t)^\top\,T_k^{-1}\,(-\nu'_k(t))\Bigr]\\[1mm]
		& = \frac{1}{2}\Bigl[\nu'_k(t)^\top\,T_k^{-1}\,z(x,t) + z(x,t)^\top\,T_k^{-1}\,\nu'_k(t)\Bigr] + \frac{1}{2}\,z(x,t)^\top\,T_k^{-1}\,T'_k\,T_k^{-1}\,z(x,t)\\[1mm]
		& = z(x,t)^\top\,T_k^{-1}\,\nu'_k(t) + \frac{1}{2}\,z(x,t)^\top\,T_k^{-1}\,T'_k\,T_k^{-1}\,z(x,t),
	\end{split}
	\]
	where we used the symmetry of $T_k^{-1}$. Next,
	\[
	D m_k(x,t)= m_k(x,t)\,D g_k(x,t)= -m_k(x,t)\,T_k(t)^{-1}\,z(x,t).
	\]
	For the Laplacian term,
	\[
	\Delta m_k(x,t)= m_k(x,t)\Bigl[\Delta g_k(x,t)+\lvert D g_k(x,t)\rvert^2\Bigr],
	\]
	with
	\(
	\Delta g_k(x,t)= -\operatorname{tr}\bigl(T_k(t)^{-1}\bigr) \) and \( \lvert D g_k(t,x)\rvert^2= z(x,t)^\top\,T_k(t)^{-2}z(x,t).\)
	Thus,
	\[
	-\e\,\Delta m_k(x,t)= -\e m_k(x,t)\left[-\operatorname{tr}\bigl(T_k(t)^{-1}\bigr) + z(x,t)^\top\,T_k(t)^{-2}z(x,t)\right].
	\]
For the drift term,
\[
\operatorname{div}\bigl(m_k(x,t) V_k(t)(x-M_k(t))\bigr)
= m_k\operatorname{tr}(V_k(t)) + (x-M_k(t))^\top V_k(t)^\top D m_k(x,t).
\]
\[=m_k(x,t) (\operatorname{tr}(V_k(t)) - (x-M_k(t))^\top\,V_k(t)\,T_k(t)^{-1}z(x,t)).\]
	Inserting the expressions into the FP equation, and dividing by $m_k(x,t)$, we get
	\[
	\frac{c'_k(t)}{c_k(t)} + \partial_t g_k(x,t) + \e\,\operatorname{tr}\bigl(T_k(t)^{-1}\bigr) - \e\,z(x,t)^\top\,T_k(t)^{-2}z(x,t) \]\[- \operatorname{tr}\bigl(V_k(t)\bigr) + (x-M_k(t))^\top\,V_k(t)\,T_k(t)^{-1}z(x,t) = 0.
	\]
	By writing $x-M_k(t) = z(x,t)+(\nu_k(t)-M_k(t))$, and by relying on the above obtained expression for $\partial_t g_k(x,t)$, we rewrite the equation as
	\begin{align*}
		\frac{c'_k(t)}{c_k(t)} + z(x,t)^\top\,T_k^{-1}(t)\nu'_k(t) + \frac{1}{2}\,z(x,t)^\top\,T_k^{-1}(t)T'_k(t)\,T_k^{-1}(t)z(x,t) + \\
		\e\,\operatorname{tr}(T_k^{-1}(t)) - \e\,z(x,t)^\top\,T_k^{-2}(t)z(x,t) - \operatorname{tr}(V_k(t)) \\+ \Bigl[z(x,t)^\top\,V_k(t)T_k^{-1}(t)z(x,t) +( \nu_k(t)-M_k(t))^\top\,V_k(t)T_k^{-1}(t)z(x,t)\Bigr]= 0.
	\end{align*}
	Collecting terms according to their power in $z(x,t)$, the quadratic terms yield
	\begin{align*}
	\frac{1}{2}&\,z(x,t)^\top\,T_k^{-1}(t)T'_k(t)\,T_k^{-1}(t)z(x,t) - \e\,z(x,t)^\top\,T_k^{-2}(t)z(x,t)\\
	 &+ z(x,t)^\top\,V_k(t)T_k^{-1}(t)z(x,t) = 0.
	\end{align*}
	We firstly notice that this has to hold for every $z(x,t)$. Secondly, by multiplying on the left and right by $T_k$, we obtain 
	 \[
	T'_k(t)= 2\e\,I - T_k(t)\,V_k(t)-V_k(t)\,T_k(t).
	\]
	The linear terms in $z(x,t)$ yield
	\[
	T_k^{-1}(t)\nu'_k(t) + T_k^{-1}(t)\,V_k(t)(\nu_k(t)-M_k(t))=0,
	\]
	which implies
	\[
	\nu'_k(t)= -V_k(t)\Bigl(\nu_k(t)-M_k(t)\Bigr).
	\]
	Finally, the constant terms give
	\[
	\frac{c'_k(t)}{c_k(t)} + \e\,\operatorname{tr}(T_k^{-1}(t)) - \operatorname{tr}(V_k(t)) = 0.
	\]
	This relation is automatically verified by differentiating 
	\[
	c_k(t)= \frac{1}{\sqrt{(2\pi)^d\det T_k(t)}},
	\]
	indeed, by exploiting the identity \(\frac{d}{dt}\ln\det T_k(t)= \operatorname{tr}\Bigl(T_k(t)^{-1}\,T'_k(t)\Bigr)\), and \(
	T'_k(t)= 2\e\,I - T_k(t)\,V_k(t)-V_k(t)\,T_k(t)
	\), we have
	\[
	\frac{c'_k(t)}{c_k(t)}=-\frac{1}{2}\frac{d}{dt}\ln\det T_k(t)
	= -\frac{1}{2}\operatorname{tr}\Bigl(T_k(t)^{-1}\,T'_k(t)\Bigr)=\operatorname{tr}\bigl(V_k(t)\bigr)- \e\,\operatorname{tr}\bigl(T_k(t)^{-1}\bigr).
	\] 
\end{proof}
\begin{proof}[Proof of Prop. \ref{prop:equal_data}]
By the definition of \(V_k(t)\) through the Sylvester equation,
\[
\Sigma_k(t)V_k(t)+V_k(t)\Sigma_k(t)
=
2\varepsilon I-\Sigma_k'(t).
\]
Equivalently,
\[
\Sigma_k'(t)
=
2\varepsilon I-\Sigma_k(t)V_k(t)-V_k(t)\Sigma_k(t),
\]
which is precisely the matrix differential equation satisfied by \(T_k(t)\).
Since \(T_k(0)=\Sigma_k(0)\), uniqueness for the corresponding linear
matrix differential equation yields
\[
T_k(t)=\Sigma_k(t)
\qquad\text{for all }t\in[0,T].
\]
Moreover, from
\[
M_k(t)=\mu_k(t)+V_k(t)^{-1}\mu_k'(t)
\]
we obtain
\[
\mu_k'(t)
=
-V_k(t)\bigl(\mu_k(t)-M_k(t)\bigr).
\]
Thus, \(\mu_k\) satisfies the same differential equation as \(\nu_k\).
Since \(\nu_k(0)=\mu_k(0)\), uniqueness gives
\[
\nu_k(t)=\mu_k(t)
\qquad\text{for all }t\in[0,T].
\]
By \eqref{equal_mean}, it follows that
\begin{align*}
	\E_{m(\cdot,t)}[X]&=\sum_{k=1}^K\a_k(t)\E_{m_k(\cdot,t)}[X]=\sum_{k=1}^K\a_k(t) \mu_k(t)\\
	&=\sum_{k=1}^K \a_k(t)\frac{\int_{\R^d}x\gamma_{k}(x,t)f(x,t)dx}{\a_k(t)} 
	=\int_{\R^d} xf(x,t)dx=\E_{f(\cdot,t)}[X]
\end{align*}
since $\sum_{k=1}^K \gamma_{k}(x,t)=1$. \\
To prove  $\Var_{m(\cdot,t)}[X] =\Var_{f(\cdot,t)}[X] $, we first claim that  
\begin{equation}\label{eq_moment}
	\int_{\R^d } xx^\top m(x,t) dx = \int_{\R^d} xx^\top f(x,t) dx.
\end{equation} 
Indeed, recalling that $T_k=\Sigma_k$,  multiplying by $\alpha_k=\int_{\mathbb{R}^d}\g_k(x)f(x)dx $        \eqref{emmfg_var}   and summing over $k$, by the identity
\begin{align*}
	&\sum_{k=1}^{K} \alpha_k \int_{\mathbb{R}^d}(x- \mu_k(t))(x-\mu_k(t))^\top m_k (x,t) dx\\
	 &= \sum_{k=1}^{K} \alpha_k (t)\dfrac{\int_{\R^d }(x - \mu_k(t))(x-\mu_k(t))^\top \gamma_k (x,t) f(x,t) dx}{\int_{\mathbb{R}^d}\g_k(x,t)f(x,t)dx}, 
\end{align*}
we easily get  
\begin{equation*}
	\int_{\R^d} xx^\top m(x,t) dx - \sum_{k=1}^{K}\alpha_k(t) \mu_k(t)\mu_k^\top(t)= \int_{\R^d} xx^\top f(x,t) dx - \sum_{k=1}^{K}\alpha_k(t)\mu_k(t)\mu_k^\top(t) 
\end{equation*}
and therefore the claim.	
Hence, by the $\E_{m(\cdot,t)}[X]=\E_{f(\cdot,t)}[X]=:\mu(t)$, we get
\begin{align*}
	\Var_{m(\cdot,t)}[X] =\int_{\R^d} xx^\top m(x,t) dx-\mu(t)\mu(t)^\top \\
	=\int_{\R^d} xx^\top f(x,t) dx- \mu(t)\mu(t)^\top =\Var_{f(\cdot,t)}[X] .
\end{align*}
\end{proof}
\begin{proof}[Proof of Prop. \ref{prop:log_lik_1} and \ref{prop:log_lik2}]
We work with a general kernel \(g\) to prove \ref{prop:log_lik2}. 
The aim is to find the components of the mixture that maximize \[
\cL_g=\sum_{k=1}^K
\int_{\R}g(t-s)\int_{\R^d}\gamma_k(x,s)\ln\bigl[\beta_k\,\mathcal N(x\mid\rho_k,\Delta_k)\bigr]\,f(x,s)\,dx\,ds.
\]
We write explicitly
		\[
		\ln\bigl[\beta_k\,\mathcal N(x\mid\rho_k,\Delta_k)\bigr]
		=\ln\beta_k -\tfrac d2\ln(2\pi)-\tfrac12\ln\det \Delta_k -\tfrac12(x-\rho_k)^T \Delta_k^{-1}(x-\rho_k),
		\] and introduce multiplier $\lambda$ for $\sum_k\beta_k=1$:
		\[
		L = \cL_g + \lambda\bigl(1-\sum_{k=1}^K\beta_k\bigr),
		\] we have three different optimization steps.\\
		Let us optimize w.r.t. $\beta_k$. Differentiating:
\[
\frac{\partial L}{\partial\beta_k}
=\int_{\R} g(t-s)\int_{\R^d}\gamma_k(x,s)\frac1{\beta_k}f(x,s)\,dx\,ds -\lambda =0\]
\[\Longrightarrow\;
\beta_k
=\frac{1}{\lambda}\int_{\R} g(t-s)\int_{\R^d}\gamma_k(x,s)f(x,s)\,dx\,ds.
\]
Imposing $\sum_k\beta_k=1$ gives
\[
\lambda=\sum_{j=1}^K\int_{\R} g(t-s)\int_{\R^d}\gamma_j(x,s)f(x,s)\,dx\,ds
\;=\;\int_{\R} g(t-s)\Biggl(\int_{\R^d} f(x,s)dx\Biggr)ds=1,
\]
so
\[\beta_k
=\displaystyle\int_{\R} g(t-s)\int_{\R^d}\gamma_k(x,s)f(x,s)\,dx\,ds.
\]
Only the quadratic term depends on $\rho_k$.  We have
\[
\frac{\partial L}{\partial\rho_k}
=-\frac12
\int_{\R} g(t-s)\frac{\partial}{\partial\rho_k}
\int_{\R^d}\gamma_k(x,s)(x-\rho_k)^\top \Delta_k^{-1}(x-\rho_k)\,f(x,s)\,dx\,ds
= 
\]
\[
=\int_{\R} g(t-s)\int_{\R^d}\gamma_k(x,s)\Delta_k^{-1}(x-\rho_k)f(x,s)\,dx\,ds=0
\;
\]
Hence
\[
\rho_k
=\frac{\displaystyle\int_{\R} g(t-s)\,\int_{\R^d} x \gamma_k(x,s)f(x,s)\,dxds}{\displaystyle\int_{\R} g(t-s)\int_{\R^d}\gamma_k(x,s)f(x,s)\,dx\,ds}.
\]
		Finally, w.r.t. $\Delta_k$.
		Using matrix derivatives (see \cite{MatrixCookbook}, Eq. 57)
		\[
		\frac{\partial}{\partial \Delta_k}\ln\det \Delta_k = \Delta_k^{-1}  ,
		\]\[
		\frac{\partial}{\partial \Delta_k}(x-\rho_k)^\top \Delta_k^{-1}(x-\rho_k)
		= -\Delta_k^{-1}(x-\rho_k)(x-\rho_k)^\top \Delta_k^{-1} \quad\textrm{(see \cite{MatrixCookbook}, Eq. 61)}
		\]
we have
		\[
		\frac{\partial L}{\partial \Delta_k}
		=-\frac{1}{2}\int_{\R} g(t-s)\int_{\mathbb{R}^d}\gamma_k(x,s)\bigl[\Delta_k^{-1} - \Delta_k^{-1}(x-\rho_k)(x-\rho_k)^\top \Delta_k^{-1}\bigr]f(x,s)\,dxds=0.
		\]
By multiplying on left and right by $\Delta_k$, 
\[
0=\int_{\R} g(t-s)\int_{\R^d}\gamma_k(x,s)\bigl[(x-\rho_k)(x-\rho_k)^\top-\Delta_k\bigr]f(x,s)\,dx\,ds
\;
\]
so, by adding and subtracting $\mu_k(s)$ inside the inner integral, we obtain
\begin{equation}\label{eq_cov_max}
\begin{split}
\Delta_k
&=\frac{\displaystyle\int_{\R} g(t-s)\,\int_{\R^d} (x-\rho_k)(x-\rho_k)^\top  \gamma_k(x,s)f(x,s)\,dxds}{\displaystyle\int_{\R} g(t-s)\int_{\R^d}\gamma_k(x,s)f(x,s)\,dx\,ds} \\[8pt]
&= \frac{\displaystyle\int_{\R} g(t-s)\alpha_k(s)\Bigl[ \Sigma_k(s) + \bigl(\mu_k(s)-\rho_k\bigr)\bigl(\mu_k(s)-\rho_k\bigr)^\top \Bigr] ds}{\displaystyle\int_{\R} g(t-s)\alpha_k(s)\,ds} \\[8pt]
&= \frac{\bigl(g * (\alpha_k\Sigma_k)\bigr)(t)}{\bigl(g * \alpha_k\bigr)(t)} + \frac{\Bigl(g * \bigl(\alpha_k(\mu_k-\rho_k)(\mu_k-\rho_k)^\top\bigr)\Bigr)(t)}{\bigl(g * \alpha_k\bigr)(t)}.
\end{split}
\end{equation}
At the optimum with respect to the mean parameter, we have
\[
\rho_k=\widetilde\mu_k(t).
\]
Therefore, \eqref{eq_cov_max} gives the exact covariance maximizer
\[
\Delta_k^\star(t)
=
\widetilde\Sigma_k(t)+R_k(t),
\]
where
\[
R_k(t)
=
\frac{
\bigl(
g*
\bigl[
\alpha_k
(\mu_k-\widetilde\mu_k(t))
(\mu_k-\widetilde\mu_k(t))^\top
\bigr]
\bigr)(t)
}{
(g*\alpha_k)(t)
}.
\]
In particular, \(R_k(t)\) is positive semidefinite.

To estimate this term, introduce the probability measure
\[
dP_{k,t}(s)
=
\frac{g(t-s)\alpha_k(s)}
{(g*\alpha_k)(t)}\,ds.
\]
Then
\[
\widetilde\mu_k(t)
=
\int \mu_k(s)\,dP_{k,t}(s)
\]
and the standard covariance identity gives
\[
R_k(t)
=
\frac12
\iint
\bigl(\mu_k(s)-\mu_k(r)\bigr)
\bigl(\mu_k(s)-\mu_k(r)\bigr)^\top
\,dP_{k,t}(s)\,dP_{k,t}(r).
\]
Hence,
\[
\begin{aligned}
\|R_k(t)\|_2
&\leq
\frac12
\iint
\|\mu_k(s)-\mu_k(r)\|_2^2
\,dP_{k,t}(s)\,dP_{k,t}(r)\\
&\leq
\frac12 L_{\mu,k}(t)^2\tau^2,
\end{aligned}
\]
because the support of the kernel has diameter \(\tau\).

Finally, the covariance-dependent part of the objective is strictly
concave in the precision matrix \(P_k=\Delta_k^{-1}\). Therefore,
\(\Delta_k^\star=\widetilde\Sigma_k+R_k\) is the unique exact covariance
maximizer. The reduced choice
$
\Delta_k=\widetilde\Sigma_k
$
is associated with an error bounded by
\[
\|\Delta_k^\star-\widetilde\Sigma_k\|_2
\leq
\frac12 L_{\mu,k}(t)^2\tau^2.
\]
\end{proof}
\begin{proof}[Proof of Prop. \ref{prop:coeffmixt_kernel}]
The identities $\nu_k(t)=\widetilde\mu_k(t)$ and $T_k(t)=\widetilde\Sigma_k(t)$ follow from the same uniqueness argument used in the proof of
Proposition~\ref{prop:equal_data}, with \(\mu_k\) and \(\Sigma_k\) replaced by
\(\widetilde\mu_k\) and \(\widetilde\Sigma_k\), respectively, and using the
corresponding initial conditions.
\[
\begin{aligned}
\E_{m(\cdot,t)}[X]
&=\sum_{k=1}^K\tilde\a_k(t)\,\tilde\mu _k(t)=\sum_{k=1}^K\tilde\a_k(t)\,\frac{(g*(\mu_k\,\alpha_k))(t)}{\tilde\a_k(t)}=\\
&=\sum_{k=1}^K\int_{\R}g(t-s)\int_{\R^d}x\,\gamma_k(x,s)\,f(x,s)\,dx\,ds=\\
&=\int_{\R}g(t-s)\int_{\R^d}x\Biggl(\sum_{k=1}^K\gamma_k(x,s)\Biggr)\,f(x,s)\,dx\,ds=\\
&=\int_{\R}g(t-s)\int_{\R^d}x\,f(x,s)\,dx\,ds
=\bigl(g*\E_{f}[X]\bigr)(t).
\end{aligned}
\]
\[
\E_{m(\cdot,t)}[XX^T]
\;-\sum_{k=1}^{K} \tilde\alpha_k(t)\,\tilde\mu_k(t)\,\tilde\mu_k(t)^T= \]
\[
= \sum_{k=1}^{K} \tilde\alpha_k(t) \int_{\mathbb{R}^d}(x-\tilde\mu_k(t))(x-\tilde\mu_k(t))^T \,m_k(x,t)\,dx= \sum_{k=1}^{K} \tilde\alpha_k(t)\,\tilde\Sigma_k(t)=\]
\[= \sum_{k=1}^{K} \int_{\R} g(t-s)\Bigl[\,x x^T\,\gamma_k(x,s)\,f(x,s)\;-\;\mu_k(s)\mu_k(s)^T\!\int_{\mathbb{R}^d}\gamma_k(x,s)\,f(x,s)\,dx\Bigr] ds=\]
\[
= \bigl(g*\E_f[XX^T]\bigr)(t)\;-\;\bigl(g*\sum_k \alpha_k\,\mu_k\mu_k^T\bigr)(t).
\]
Therefore, the covariance under $m$ is
\begin{align*}
\operatorname{Cov}_{m(\cdot,t)}[X] &= \mathbb{E}_{m(\cdot,t)}[XX^\top] - \mathbb{E}_{m(\cdot,t)}[X]\mathbb{E}_{m(\cdot,t)}[X]^\top \\
&= \mathbb{E}_{m(\cdot,t)}[XX^\top] - \bigl(g*\mathbb{E}_f[X]\bigr)(t)\bigl(g*\mathbb{E}_f[X]\bigr)(t)^\top \\
&= \bigl(g*\mathbb{E}_f[XX^\top]\bigr)(t) - \bigl(g*\mathbb{E}_f[X]\bigr)(t)\bigl(g*\mathbb{E}_f[X]\bigr)(t)^\top \\
&\quad + \sum_{k=1}^{K} \tilde\alpha_k(t)\,\tilde\mu_k(t)\,\tilde\mu_k(t)^\top - \Bigl(g*\sum_{k=1}^{K} \alpha_k\,\mu_k\mu_k^\top\Bigr)(t).
\end{align*}
\end{proof}

\section{Computational approach and discretization}
\label{sec:computational-approach}

We describe the numerical procedures used for the parametric Gaussian
formulation and for the fully non-parametric mean-field-game formulation.
Throughout this section, \(k=1,\ldots,K\) denotes the component index,
\(n=0,\ldots,N_T\) the physical-time index, \(i=1,\ldots,M\) the
spatial-grid index, \(q\) the outer fixed-point iteration, and \(r\) the
pseudo-time iteration used to solve the stationary Hamilton--Jacobi
equation.

The observed probability density at time \(t_n\) is denoted by
\[
    f^n(x):=f(x,t_n),
\]
whereas \(m_k^n\) and \(\alpha_k^n\) denote, respectively, the density and
the mixture weight of the \(k\)-th component.

\paragraph{Parametric Gaussian formulation.}

In the Gaussian setting, the numerical method follows an EM-type
estimation--evolution procedure. At each physical time level, the
responsibilities and the Gaussian parameters are first updated in an
E-step. The component densities are then advanced through the corresponding
Fokker--Planck equations in an M-step.

\begin{enumerate}

\item[\textbf{0.}]
\textbf{Space and time discretization.}

The time interval \([0,T]\) is partitioned uniformly as
\[
    t_n=n\Delta t,
    \qquad n=0,\ldots,N_T,
    \qquad T=N_T\Delta t.
\]
The spatial domain \(\Omega\subset\mathbb{R}^d\) is partitioned into
cells \(C_i\), \(i=1,\ldots,M\). We denote by
\[
    \Delta x:=\max_{1\leq i\leq M}\operatorname{diam}(C_i)
\]
the characteristic mesh size, by \(x_i\) the quadrature point associated
with \(C_i\), and by \(\Delta V_i=|C_i|\) its volume.
The observed and component densities are represented by their values at the
cell barycentres.

\item[\textbf{1.}]
\textbf{Initialization by static soft clustering.}

The initial component densities and mixture weights are computed by the
static soft-clustering method introduced in~\cite{accd}. Starting from a
non-degenerate initial assignment, the responsibilities
\(\gamma_{k,i}^0\) are used to evaluate
\[
    \alpha_k^0
    =
    \sum_{i=1}^{M}     \Delta V_i
    \gamma_{k,i}^0 f_i^0,
\quad
    \mu_k^0
    =
    \frac{1}{\alpha_k^0}
    \sum_{i=1}^{M}     \Delta V_i
    x_i\gamma_{k,i}^0 f_i^0,
\]
and
\[
    \Sigma_k^0
    =
    \frac{1}{\alpha_k^0}
    \sum_{i=1}^{M} \Delta V_i
    (x_i-\mu_k^0)(x_i-\mu_k^0)^\top
    \gamma_{k,i}^0 f_i^0,
\]
where \(f_i^0=f(x_i,t_0)\).

The initial component density is then defined as the discretization of the
Gaussian distribution
\[
    m_k^0(x)
    =
    \mathcal{N}(x\mid\mu_k^0,\Sigma_k^0).
\]
The responsibilities, weights, means, and covariance matrices are updated
until convergence. This procedure provides an initialization that is
consistent with both the discrete observations and the continuous densities
used in the subsequent Fokker--Planck evolution, while avoiding degenerate
or poorly separated initial profiles.

\item[\textbf{2.}]
\textbf{E-step: update of the latent variables and Gaussian parameters.}

At time \(t_n\), the responsibilities are evaluated from the component
densities and weights available at the preceding time level:
\[
    \gamma_{k,i}^{n}
    =
    \frac{
        \alpha_k^{n-1}m_k^{n-1}(x_i)
    }{
        \displaystyle
        \sum_{\ell=1}^{K}
        \alpha_\ell^{n-1}m_\ell^{n-1}(x_i)
    }.
    \label{eq:discrete-gaussian-responsibilities}
\]
The mixture weights, means, and covariance matrices are then approximated
by the quadrature formulas
\[
    \alpha_k^n
    =
    \sum_{i=1}^{M}\Delta V_i
    \gamma_{k,i}^n f_i^n,
\quad
    \mu_k^n
    =
    \frac{1}{\alpha_k^n}
    \sum_{i=1}^{M}\Delta V_i
    x_i\gamma_{k,i}^n f_i^n,
    \label{eq:discrete-gaussian-means}
\]
and
\[
    \Sigma_k^n
    =
    \frac{1}{\alpha_k^n}
    \sum_{i=1}^{M}\Delta V_i
    (x_i-\mu_k^n)(x_i-\mu_k^n)^\top
    \gamma_{k,i}^n f_i^n,
    \label{eq:discrete-gaussian-covariances}
\]
where \(f_i^n=f(x_i,t_n)\).

These quantities determine the drift field \(B_k^n\) entering the
Fokker--Planck equation in the subsequent M-step. To construct the drift, the time derivatives of the Gaussian moments are
approximated by backward differences,
\[
    \dot{\mu}_k^n
    \approx
    \frac{\mu_k^n-\mu_k^{n-1}}{\Delta t},
    \qquad
    \dot{\Sigma}_k^n
    \approx
    \frac{\Sigma_k^n-\Sigma_k^{n-1}}{\Delta t}.
\]
For the time-averaged formulations, the same formulas are applied to the
regularized quantities
\(\widetilde{\mu}_k^n\) and \(\widetilde{\Sigma}_k^n\). The matrix \(V_k^n\) is obtained from the discrete Sylvester equation
\[
    \Sigma_k^n V_k^n
    +
    V_k^n\Sigma_k^n
    =
    2\varepsilon I-\dot{\Sigma}_k^n.
\]
When \(\Sigma_k^n\) and \(\dot{\Sigma}_k^n\) commute, this reduces to
\[
    V_k^n
    =
    \frac12
    (\Sigma_k^n)^{-1}
    \left(
        2\varepsilon I-\dot{\Sigma}_k^n
    \right).
\]
The centre of the affine drift is then
\[
    M_k^n
    =
    \mu_k^n
    +
    (V_k^n)^{-1}\dot{\mu}_k^n,
\]
and
\[
    B_k^n(x)
    =
    V_k^n(x-M_k^n).
\]

\item[\textbf{3.}]
\textbf{M-step: Fokker--Planck evolution.}

For each component, the density evolution is governed by
\[
    \partial_t m_k
    =
    \operatorname{div}\!\left(B_km_k\right)
    +
    \varepsilon\Delta m_k.
    \label{eq:gaussian-fp-general}
\]
The equation is approximated by the semi-Lagrangian scheme described
in~\cite{cs18}. In algebraic form, the update can be written as
\[
    \mathbf{m}_k^n
    =
    A_{\Delta x,\Delta t}^{\varepsilon}[B_k^n]
    \mathbf{m}_k^{n-1},
    \label{eq:gaussian-discrete-evolution}
\]
where
\[
    \mathbf{m}_k^n
    =
    \left(
        m_k^n(x_1),\ldots,m_k^n(x_M)
    \right)^\top
\]
and \(A_{\Delta x,\Delta t}^{\varepsilon}[B_k^n]\) is the discrete evolution
operator accounting for both transport by \(B_k^n\) and diffusion with
coefficient \(\varepsilon\).

\item[\textbf{4.}]
\textbf{Fixed-point procedure for the symmetric regularization.}

For the instantaneous and causally regularized models, the E-step and
M-step can be performed sequentially at each physical time level. The
symmetric temporal regularization is instead non-causal, since the update at
time \(t_n\) depends on configurations at both earlier and later times.
Consequently, the complete time-dependent solution must be determined
through an outer fixed-point iteration.

We denote the fixed-point iterate by the superscript \([q]\). Given
\[
    \left\{
        m_k^{n,[q]},\alpha_k^{n,[q]}
    \right\}_{k,n},
\]
the responsibilities and symmetrically averaged quantities are recomputed
over the entire interval, after which the E-step and M-step generate the
next iterate
\[
    \left\{
        m_k^{n,[q+1]},\alpha_k^{n,[q+1]}
    \right\}_{k,n}.
\]
The iteration is stopped when
\[
    \max_{\substack{k=1,\ldots,K\\n=0,\ldots,N_T}}
    \left|
        \alpha_k^{n,[q+1]}
        -
        \alpha_k^{n,[q]}
    \right|
    <
    \mathrm{tol}_{\mathrm{FP}},
    \label{eq:gaussian-fixed-point-stopping}
\]
where \(\mathrm{tol}_{\mathrm{FP}}>0\) is a prescribed tolerance.

Stabilization of the weights is used as a practical stopping criterion.
In the reported experiments, the variations of the means, covariance
matrices, and component densities were also monitored to verify convergence
of the complete state.

\end{enumerate}

\begin{remark}
In the algorithm we define a sequential first-order discretization: the
component configuration available at \(t_{n-1}\) is used to assign the
observation at \(t_n\), and the resulting moments determine the drift over
the subsequent evolution step.

In the Gaussian setting, explicitly solving the complete PDE system is not
strictly necessary for practical computations. Owing to the theoretical
equivalence established in the previous sections, the corresponding
finite-dimensional EM updates provide a substantially less expensive
alternative.

The numerical solution of the PDE system nevertheless serves two purposes:
it validates the continuous-time MFG interpretation and provides the
computational framework required for the non-parametric extension. In the
latter setting, the component densities are not restricted to prescribed
Gaussian shapes, and the coupled Hamilton--Jacobi--Fokker--Planck system must
therefore be solved explicitly.
\end{remark}

\paragraph{Non-parametric mean-field-game formulation}

In the non-parametric formulation, the component densities are not assumed
to belong to a prescribed statistical family. At each time level, the
complete Hamilton--Jacobi--Fokker--Planck system is solved for every
component.

The reconstructed mixture at time \(t_{n-1}\) is
\[
    m^{n-1}(x)
    =
    \sum_{\ell=1}^{K}
    \alpha_\ell^{n-1}m_\ell^{n-1}(x),
    \label{eq:test1-mfg-mixture}
\]
and the responsibilities are defined by
\[
    \gamma_k^n(x)
    =
    \frac{
        \alpha_k^{n-1}m_k^{n-1}(x)
    }{
        \displaystyle
        \sum_{\ell=1}^{K}
        \alpha_\ell^{n-1}m_\ell^{n-1}(x)
    }.
    \label{eq:test1-mfg-responsibilities}
\]

\begin{itemize}

\item[\textbf{0.}]\textbf{MFG coupling.}

The information-theoretic coupling introduced in
Section~\ref{sec:MFG} is given by
\[
    F_k^n(x)
    =
    \eta_k\log m_k^{n-1}(x)
    +
   (1-\eta_k)\alpha_k^{n-1}
    \log\left(
        \frac{m^{n-1}(x)}
             {f^n(x)}
    \right).
    \label{eq:test1-mfg-coupling}
\]
The first term is a component-dependent entropic contribution. The second
term is associated with the first variation of the reverse
Kullback--Leibler divergence
\[
    D_{\mathrm{KL}}
    (m\,\|\,f)
    =
    \int_\Omega
    m(x)
    \log\left(
        \frac{m(x)}
             {f(x)}
    \right)\,dx
\]
and acts as a fidelity term between the reconstructed mixture and the
observed density.

For numerical stability, every density appearing inside a logarithm or in
the denominator of the responsibilities is replaced by
\[
    \max\{m,\delta_{\mathrm{num}}\}, \hbox{ or }  \max\{f,\delta_{\mathrm{num}}\},
\]
where \(\delta_{\mathrm{num}}>0\) is a small numerical threshold. Moreover,
the spatial mean of \(F_k^n\) is subtracted before solving the
Hamilton--Jacobi equation. This modification affects only the ergodic
constant and leaves the optimal drift unchanged.

In the numerical experiments, the same coupling parameter is used for all
components, namely \(\eta_k\equiv\eta\).

\item[\textbf{1.}]\textbf{Stationary Hamilton--Jacobi equation.}

For the quadratic Hamiltonian
\[
    H(p)=\frac{1}{2}|p|^2,
\]
the stationary Hamilton--Jacobi problem for the \(k\)-th component is
\[
    -\varepsilon\Delta u_k^n
    +
    \frac{1}{2}\left|\nabla u_k^n\right|^2
    +
    \lambda_k^n
    =
    F_k^n
    \qquad\text{in }\Omega.
    \label{eq:test1-stationary-hj}
\]
Homogeneous Neumann boundary conditions are imposed:
\[
    \nabla u_k^n\cdot\boldsymbol{\nu}=0
    \qquad\text{on }\partial\Omega,
\]
where \(\boldsymbol{\nu}\) denotes the outward unit normal. The additive
indeterminacy of the value function is removed by imposing
\[
    \min_{x\in\Omega}u_k^n(x)=0.
\]

The stationary equation is solved through a pseudo-time iteration. The
diffusion term is treated implicitly, whereas the Hamiltonian and the
coupling are evaluated explicitly. In \(d\) dimensions, the monotone
Godunov approximation of the quadratic Hamiltonian is
\[
    \widehat H_i(U)
    =
    \frac{1}{2}
    \sum_{j=1}^{d}
    \left[
        \max\!\left(D_j^-U_i,0\right)^2
        +
        \min\!\left(D_j^+U_i,0\right)^2
    \right].
    \label{eq:test1-godunov}
\]
At pseudo-time iteration \(r\), the numerical update is
\[
    \left(
        I+\Delta r A_{\varepsilon}
    \right)
    U_k^{n,(r+1)}
    =
    U_k^{n,(r)}
    -
    \Delta r
    \left[
        \widehat H\!\left(U_k^{n,(r)}\right)
        -
        F_k^n
    \right],
    \label{eq:test1-hj-iteration}
\]
where \(A_\varepsilon\) denotes the finite-difference approximation of
the elliptic operator \(-\varepsilon\Delta\), including the homogeneous
Neumann boundary conditions.

After every pseudo-time iteration, the minimum of \(U_k^{n,(r+1)}\) is
subtracted. The iteration is terminated when
\[
    \left\|
        U_k^{n,(r+1)}
        -
        U_k^{n,(r)}
    \right\|_\infty
    <
    \mathrm{tol}_{\mathrm{HJ}}.
    \label{eq:hj-stopping-criterion}
\]
The additive shift removed at each iteration, divided by \(\Delta r\),
also provides an approximation of the ergodic constant \(\lambda_k^n\).

\item[\textbf{2.}]\textbf{Fokker--Planck update.}

For the quadratic Hamiltonian, the optimal velocity field is
\[
    b_k^n(x)
    =
    -\nabla u_k^n(x).
    \label{eq:mfg-optimal-velocity}
\]
The component density is advanced by solving
\[
    \frac{
        m_k^n-m_k^{n-1}
    }{\Delta t}
    -
    \varepsilon\Delta m_k^n
    +
    \operatorname{div}
    \left(
        b_k^n m_k^n
    \right)
    =
    0.
    \label{eq:test1-fp-discretization}
\]
The associated no-flux boundary condition is
\[
    \left(
        b_k^n m_k^n
        -
        \varepsilon\nabla m_k^n
    \right)
    \cdot\boldsymbol{\nu}
    =
    0
    \qquad\text{on }\partial\Omega.
\]

Both diffusion and transport are treated implicitly. The transport term is
approximated by a first-order upwind discretization selected according to the
sign of the numerical velocity. At every physical time step, this produces
one sparse linear system of dimension \(M\) for each component.

\item[\textbf{3.}]\textbf{Weights update.} The weights associated with the current observation are updated as
\[
    \alpha_k^n=
    \sum_{i=1}^{M}
    \Delta V_i\gamma_{k,i}^n f_i^n.
    \label{eq:np-mfg-weights}
\]
After the update, the weights are renormalized so that
\[
    \sum_{k=1}^{K}\alpha_k^n=1.
\]

\end{itemize}

Once the mixture, responsibilities, and coupling terms have been assembled,
the Hamilton--Jacobi and Fokker--Planck problems are independent with
respect to \(k\). The component-wise problems can therefore be solved in
parallel.


\end{document}